\documentclass[11pt]{amsart}
\usepackage[margin=1in]{geometry}
\usepackage{amsmath,amssymb,amsthm}
\usepackage{booktabs,longtable,array}
\usepackage[hidelinks]{hyperref}

\newtheorem{theorem}{Theorem}[section]
\newtheorem{lemma}[theorem]{Lemma}
\newtheorem{proposition}[theorem]{Proposition}
\newtheorem{corollary}[theorem]{Corollary}
\newtheorem{conjecture}[theorem]{Conjecture}
\newtheorem{question}[theorem]{Question}
\theoremstyle{definition}
\newtheorem{definition}[theorem]{Definition}
\newtheorem{example}[theorem]{Example}
\theoremstyle{remark}
\newtheorem{remark}[theorem]{Remark}

\makeatletter
\@namedef{subjclassname@2020}{\textup{2020} Mathematics Subject Classification}
\makeatother

\newcommand{\Z}{\mathbb{Z}}

\newcommand{\bg}{\bar{g}}
\newcommand{\bfg}{\mathbf{g}}
\newcommand{\Dic}{\mathrm{Dic}}
\newcommand{\cm}{\checkmark}
\newcommand{\xm}{$\times$}
\newcommand{\Sig}{\Sigma}

\title{Partial sequencings and other combinatorial aspects of large groups}
\author{Mohammad Javaheri}
\address{Department of Mathematics and Statistics, Siena University\\ 515 Loudon Road, Loudonville, NY 12211, USA}
\email{mjavaheri@siena.edu}
\date{Revised, November 2026}
\subjclass[2020]{05E16, 05B15, 20D60}
\keywords{Sequenceable, harmonious, R-sequenceable, symmetric sequencing, partial sequencing}

\begin{document}

\begin{abstract}
A \emph{partial sequencing} of length $k$ in a finite group $G$ is a sequence $g_0,g_1,\dots,g_{k-1}$ of distinct elements of $G$ with $g_0=1_G$ such that the partial products $g_0,\, g_0g_1,\,\dots,\, g_0g_1\cdots g_{k-1}$ are also distinct. Partial R-sequencings and partial harmonious sequences are defined similarly. Using the recent work of M\"uyesser and Pokrovskiy, we determine when such partial sequences extend to full ones in sufficiently large groups. In particular, if $G$ is sufficiently large and $k\le\sqrt{|G|}$, then a partial sequencing of length $k$ extends to a sequencing of $G$ if and only if some element of $G$ that is not one of its partial products has the same image in the abelianization as the product of all elements of $G$. This holds automatically when the commutator subgroup has more than $k$ elements, and we show that it can fail otherwise. Using the fact that every group of odd order has a symmetric harmonious sequence, we also prove that $\Z_{2^k}\times H$ is doubly R-sequenceable for every group $H$ of odd order, unless it is $\Z_2$ (in particular, every abelian group other than $\Z_2$ is), that $K\times H$ is symmetric sequenceable whenever $K$ is, that all sufficiently large groups are doubly sequenceable, doubly R-sequenceable, doubly harmonious and doubly $\#$-harmonious, and we verify these four properties for all groups of order $4$ to $12$ and for several small dihedral and dicyclic groups.
\end{abstract}

\maketitle

%=====================================================================
\section{Introduction}
%=====================================================================

In 1961, motivated by the study of Latin squares, Gordon defined the notion of a sequenceable group~\cite{Gordon}. A finite group $G$ of order $n>1$ is \emph{sequenceable} if there exists an ordering $b_1,b_2,\dots,b_n$ of its elements with $b_1=1_G$, the identity element of $G$, such that the partial products $a_i=b_1\cdots b_i$, $1\le i\le n$, are all distinct. Such sequencings provide a constructive method to generate row-complete Latin squares, with applications in combinatorial design theory.

Gordon~\cite{Gordon} also proved that a finite abelian group is sequenceable if and only if it is \emph{binary}, i.e.\ has a unique element of order~$2$. It is also known that the dihedral groups of orders $6$ and $8$ and the quaternion group $Q_8$ are not sequenceable. Isbell and Li~\cite{Isbell,Li} proved that all dihedral groups of order at least $10$ are sequenceable, and Anderson and Ihrig~\cite{AI93} proved that every solvable binary group, except $Q_8$, is sequenceable. These results, together with extensive computational evidence, underpin Keedwell's conjecture~\cite{Keedwell83b}: every non-abelian group of order at least $10$ is sequenceable. The strongest confirmation of Keedwell's conjecture was provided recently by M\"uyesser and Pokrovskiy~\cite{MP}, who proved that all sufficiently large non-abelian groups are sequenceable. We use their results to establish several other combinatorial properties of sufficiently large groups.

\subsection*{Notation}
Let $\bfg\colon g_1,\dots,g_m$ be a sequence in a group $G$. For $1\le i\le m$ let
\[
\bg_i=g_{i-1}^{-1}g_i,\qquad \hat g_i=g_{i-1}g_i,
\]
where indices are read cyclically, so that $g_0=g_m$. We use three derived sequences:
\[
\bar\bfg\colon g_1,\bg_2,\dots,\bg_m,\qquad
\check\bfg\colon \bg_1,\bg_2,\dots,\bg_m,\qquad
\hat\bfg\colon \hat g_1,\hat g_2,\dots,\hat g_m .
\]
Thus $\bar\bfg$ is the sequence of \emph{linear} quotients (preceded by $g_1$), while $\check\bfg$ and $\hat\bfg$ are the sequences of \emph{cyclic} quotients and products. If $g_1=1_G$, then $g_i=\bg_2\bg_3\cdots\bg_i$ for $i\ge2$, so $\bfg$ is the sequence of partial products of $\bar\bfg$.

\begin{definition}\label{def:main}
For a subset $A$ of a finite group $G$, let $P_k(A)$ denote the set of sequences of length $k|A|$ in $A$ that contain every element of $A$ exactly $k$ times.
\begin{enumerate}
\item[(i)] $A$ is \emph{sequenceable} if there is $\bfg\in P_1(A)$ with $g_1=1_G$ and $\bar\bfg\in P_1(A)$; then $\bar\bfg$ is a \emph{sequencing} of $A$. It is \emph{doubly sequenceable} if there is $\bfg\in P_2(A)$ with $g_1=1_G$ and $\bar\bfg\in P_2(A)$.
\item[(ii)] $A$ is \emph{R-sequenceable} if there is $\bfg\in P_1(A\setminus\{1_G\})$ with $\check\bfg\in P_1(A\setminus\{1_G\})$; then $\check\bfg$ is an \emph{R-sequencing} of $A$. It is \emph{doubly R-sequenceable} (\emph{DR-sequenceable}) if there is $\bfg\in P_2(A\setminus\{1_G\})$ with $\check\bfg\in P_2(A\setminus\{1_G\})$.
\item[(iii)] $G$ is \emph{symmetric sequenceable} if there is $\bfg\in P_1(G)$ with $g_1=1_G$ such that $\bar\bfg\in P_1(G)$ and $\bg_i\bg_{m+2-i}=1_G$ for all $2\le i\le m$, where $m=|G|$.
\item[(iv)] $G$ is \emph{$2$-sequenceable} if there is $\bfg\in P_1(G)$ with $g_1=1_G$ such that, for each $x\in G$,
\[
x^2\neq 1_G\ \Rightarrow\ \#\{1\le i\le m:\ \bar\bfg_i\in\{x,x^{-1}\}\}=2,\qquad
x^2=1_G\ \Rightarrow\ \#\{1\le i\le m:\ \bar\bfg_i=x\}=1,
\]
where $\bar\bfg_i$ denotes the $i$-th term of $\bar\bfg$. (The sequence $\bfg$ is a \emph{terrace}.)
\item[(v)] $A$ is \emph{harmonious} if there is $\bfg\in P_1(A)$ with $\hat\bfg\in P_1(A)$; then $\bfg$ is a \emph{harmonious sequence} of $A$. It is \emph{doubly harmonious} if there is $\bfg\in P_2(A)$ with $\hat\bfg\in P_2(A)$.
\item[(vi)] $A$ is \emph{$\#$-harmonious} if $A\setminus\{1_G\}$ is harmonious, and \emph{doubly $\#$-harmonious} if $A\setminus\{1_G\}$ is doubly harmonious.
\end{enumerate}
\end{definition}

Every sequencing is a $2$-sequencing, so sequenceable groups are $2$-sequenceable. Harmonious sequences and R-sequencings of a group $G$ give rise to complete mappings, i.e.\ bijections $\varphi\colon G\to G$ such that $x\mapsto x\varphi(x)$ is also a bijection~\cite{Paige}. Hall and Paige~\cite{HallPaige} proved that a necessary condition for a group to admit a complete mapping is that its Sylow $2$-subgroups are either trivial or non-cyclic (the \emph{Hall--Paige condition}), and conjectured that this condition is also sufficient. The Hall--Paige conjecture was reduced to the sporadic simple groups by Wilcox~\cite{Wilcox}, the sporadic groups other than $J_4$ were treated by Evans~\cite{Evans}, and the remaining case $J_4$ was settled by Bray (see~\cite{BCCSZ}).

For a group $G$ let $\pi\colon G\to G^{\mathrm{ab}}=G/[G,G]$ be the abelianization map, written additively. For a subset $X=\{x_1,\dots,x_k\}\subseteq G$ we write $\Sig X=\pi(x_1x_2\cdots x_k)$, which does not depend on the ordering. Instead of $\pi(a)+\pi(b)$ or $\pi(a)-\pi(b)$ we simply write $a+b$ and $a-b$. The following elementary facts are used repeatedly.

\begin{lemma}\label{lem:sigma}
Let $G$ be a finite group.
\begin{enumerate}
\item[(a)] $2\,\Sig G=0$.
\item[(b)] $\Sig G=0$ if $G$ satisfies the Hall--Paige condition, and $\Sig G=\pi(\iota_G)$ if $G$ is binary with unique involution $\iota_G$.
\item[(c)] If $A\subseteq G$ is harmonious, then $\Sig A=0$; if $A$ is R-sequenceable, then $\Sig A=0$; if $A$ is doubly harmonious or doubly R-sequenceable, then $2\,\Sig A=0$.
\end{enumerate}
\end{lemma}

\begin{proof}
(a) Since $x\mapsto x^{-1}$ is a bijection of $G$, $\Sig G=\sum_{x\in G}\pi(x^{-1})=-\Sig G$.
(b) If the Sylow $2$-subgroups are trivial or non-cyclic, then the product of the elements of $G$ in some order lies in $[G,G]$ by the theorem of Hall and Paige~\cite{HallPaige}, so $\Sig G=0$; this applies in particular to binary groups with generalized quaternion Sylow $2$-subgroups, for which $\iota_G\in[G,G]$. If a Sylow $2$-subgroup $S\cong\Z_{2^a}$ is cyclic and non-trivial, then by Burnside's transfer theorem $G$ has a normal $2$-complement $M$, and $[G,G]\le M$. The image of $\Sig G$ in $G/M\cong\Z_{2^a}$ is $|M|\cdot 2^{a-1}\ne 0$, so $\Sig G\neq 0$; by (a), $\Sig G$ is then the unique element of order $2$ in $G^{\mathrm{ab}}\cong M/[G,G]\times\Z_{2^a}$, which is the image of the involution of $S$.
(c) If $\bfg$ is a harmonious sequence of $A$, then $\Sig A=\sum_i\pi(\hat g_i)=2\sum_i\pi(g_i)=2\Sig A$. If $\check\bfg$ is an R-sequencing of $A$, then $\Sig(A\setminus\{1_G\})=\sum_i\pi(\bg_i)=\sum_i(\pi(g_i)-\pi(g_{i-1}))=0$. The doubly harmonious and doubly R-sequenceable cases are identical, each sum being taken twice.
\end{proof}

Table~\ref{tab:summary} and the following discussion summarize many of the known results on sequenceability and related properties, incorporating the contributions of the present article. For a comprehensive survey of sequenceable groups, $2$-sequencings (terraces), R-sequencings, harmonious sequences and related structures, see Ollis~\cite{OllisSurvey}.

\begin{itemize}
\item \textbf{Sequenceability.} Sequenceable groups include all solvable binary groups (in particular all binary abelian groups) except $Q_8$~\cite{Gordon,AI93}, dihedral groups of order $2n\ge 10$~\cite{Isbell,Li}, and all sufficiently large non-abelian groups~\cite{MP}.

\item \textbf{Symmetric sequenceability.} A symmetric sequenceable group is binary. Symmetric sequenceable groups include all solvable binary groups except $Q_8$~\cite{AI93}. Anderson conjectured in 1993 that all binary groups except $Q_8$ are symmetric sequenceable. A binary group $G$ with involution $\iota_G$ is symmetric sequenceable if and only if $G/\langle\iota_G\rangle$ is $2$-sequenceable~\cite{Anderson87}. Since $G/\langle \iota_G\rangle$ is not an elementary abelian $2$-group when $|G|>8$, Anderson's conjecture holds for sufficiently large binary groups by Bailey's conjecture for large groups (below).

\item \textbf{R-sequenceability.} R-sequenceable groups include all abelian groups satisfying the Hall--Paige condition (equivalently, all non-binary abelian groups)~\cite{AKP}, dihedral groups of order $4n\ge 4$~\cite{Keedwell83a}, dicyclic groups of order $8n\ge 16$~\cite{WangLeonard}, and all sufficiently large groups satisfying the Hall--Paige condition~\cite{MP}. By Lemma~\ref{lem:sigma}, the Hall--Paige condition is necessary.

\item \textbf{$2$-sequenceability.} $2$-sequenceable groups include all abelian groups except the elementary abelian $2$-groups of order at least $4$~\cite{Ollis05,Ollis12,OW11,OW15}, all groups of odd order~\cite{AI92}, and all sequenceable groups. Bailey~\cite{Bailey} conjectured in 1984 that all finite groups except the elementary abelian $2$-groups of order at least $4$ are $2$-sequenceable; see~\cite{Anderson92} for early computational evidence. Bailey's conjecture holds for sufficiently large groups: non-abelian ones are sequenceable~\cite{MP}, and abelian ones are covered by~\cite{OW15}.

\item \textbf{Harmoniousness.} Harmonious groups include all groups of odd order (see~\cite{BGHJ,JavSH}), all abelian groups that are neither binary nor elementary abelian $2$-groups, dihedral groups of order $8n\ge 8$~\cite{BGHJ}, dicyclic groups of order $8n\ge 16$~\cite{Wang93}, and all sufficiently large groups satisfying the Hall--Paige condition~\cite{MP}. If $G$ is a non-binary abelian group and $G\ne\Z_3$, then $G$ is $\#$-harmonious~\cite{BGHJ}. If $G$ is a binary abelian group with unique involution $\iota_G$, then $G\setminus\{\iota_G\}$ is harmonious~\cite{DJ}. Non-abelian $\#$-harmonious groups were constructed in~\cite{JavSH}: every iterated semidirect product of a ``harmoniously matched'' group by groups of odd order is both harmonious and $\#$-harmonious, which covers, e.g., $D_{16m}$ and $\Dic_{16m}$ for odd $m$, $D_{12m}$, $D_{24m}$, $\Dic_{24m}$ for $\gcd(m,6)=1$, $D_{20m}$ for $\gcd(m,10)=1$, and $K\ltimes H$ for $K\in\{A_4,S_4,D_{12},D_{16},D_{20},D_{24},\Dic_{16},\Dic_{24}\}$ and $H$ of odd order. Every group of odd order other than $1$ and $\Z_3$ is conjectured to be $\#$-harmonious~\cite[Conj.~1.7]{JavSH}. Corollaries~\ref{cor:harm-char} and~\ref{cor:HPminus} give analogous results for all sufficiently large groups.

\item \textbf{Symmetric harmoniousness.} A group is symmetric harmonious if and only if it has odd order~\cite{JavSH} (Theorem~\ref{thm:symharm}), completing earlier work of Wang and Leonard~\cite{WL94,WL95}. Consequently, if $K$ is R$^*$-sequenceable of even order, then $K\times H$ is R$^*$-sequenceable for every group $H$ of odd order~\cite{WL94,JavSH}. In this paper symmetric harmonious sequences are the key input for the DR-sequencings and symmetric sequencings of Section~\ref{sec:DR}.

\item \textbf{Double sequences.} All abelian groups are doubly sequenceable~\cite{Jav24} and doubly harmonious~\cite{DJ}, and abelian groups of order at least $4$ are doubly $\#$-harmonious~\cite{DJ}. We prove that every abelian group except $\Z_2$ is doubly R-sequenceable, and more generally that $\Z_{2^k}\times H$ is doubly R-sequenceable for every group $H$ of odd order (Theorem~\ref{thm:abelianDR}). This completes the abelian case of the following conjecture from~\cite{DJ,Jav24}.
\end{itemize}

\begin{conjecture}[\cite{DJ,Jav24}]\label{conj:double}
All groups of order at least $4$ are doubly sequenceable, doubly R-sequenceable, doubly harmonious, and doubly $\#$-harmonious.
\end{conjecture}

We verify Conjecture~\ref{conj:double} for all sufficiently large groups (Theorem~\ref{thm:large-double}), for all groups of order at most $12$, and for several further small dihedral and dicyclic groups (Section~\ref{sec:small}).

\begin{table}[ht]
\centering
\renewcommand{\arraystretch}{1.15}
\begin{tabular}{llccccccc}
\toprule
& & S & RS & H & \#H & SS & 2S & D\\
\midrule
Abelian & binary & \cm & \xm & \xm & \xm & \cm & \cm & \cm\\
order $\ge 4$ & $(\Z_2)^n$, $n\ge 2$ & \xm & \cm & \xm & \cm & \xm & \xm & \cm\\
& all others & \xm & \cm & \cm & \cm & \xm & \cm & \cm\\
\midrule
Binary & $\Dic_{8n}$, $n\ge 2$ & \cm & \cm & \cm & & \cm & \cm & C\\
non-abelian & $Z\Lambda$ & \cm & \xm & \xm & \xm & \cm & \cm & C\\
order $\ge 10$ & solvable $Q\Lambda$ & \cm & & & & \cm & \cm & C\\
& non-solvable $Q\Lambda$ & C & & & & C & C & C\\
\midrule
Dihedral & order $4n\ge 12$ & \cm & \cm & $\ast$ & $\dagger$ & \xm & \cm & C\\
order $\ge 10$ & order $4n+2\ge 10$ & \cm & \xm & \xm & \xm & \xm & \cm & C\\
\midrule
\multicolumn{2}{l}{Non-abelian, odd order} & C & & \cm & C & \xm & \cm & C\\
\midrule
Non-abelian & HP, binary & \cm & \cm & \cm & \cm & \cm & \cm & \cm\\
large order & binary, not HP & \cm & \xm & \xm & \xm & \cm & \cm & \cm\\
& HP, not binary & \cm & \cm & \cm & \cm & \xm & \cm & \cm\\
& neither & \cm & \xm & \xm & \xm & \xm & \cm & \cm\\
\bottomrule
\end{tabular}
\caption{S: sequenceable; RS: R-sequenceable; H: harmonious; \#H: $\#$-harmonious; SS: symmetric sequenceable; 2S: $2$-sequenceable; D: doubly sequenceable, doubly R-sequenceable, doubly harmonious, and doubly $\#$-harmonious. $Z\Lambda$: binary groups with cyclic Sylow $2$-subgroups; $Q\Lambda$: binary groups with non-cyclic (hence generalized quaternion) Sylow $2$-subgroups; $\Dic_{4m}=\langle a,b: a^{2m}=1,\ b^2=a^m,\ b^{-1}ab=a^{-1}\rangle$ is the dicyclic group of order $4m$; HP: the Hall--Paige condition holds. C: there is a conjecture that the answer is affirmative. $\ast$: affirmative when $8$ divides the order~\cite{BGHJ}. $\dagger$: affirmative for $D_{16m}$ ($m$ odd), $D_{12m}$, $D_{24m}$ ($\gcd(m,6)=1$) and $D_{20m}$ ($\gcd(m,10)=1$)~\cite{JavSH}. Empty cells have no explicit conjecture recorded in the literature, but the property is widely expected to hold. Section~\ref{sec:small} verifies the D column for several small groups in the rows marked C.}
\label{tab:summary}
\end{table}

\subsection*{Organization and summary of new results}
In Section~\ref{sec:symharm} we recall that every group of odd order has a symmetric harmonious sequence~\cite{JavSH} and record the properties we need. In Section~\ref{sec:DR} we use this to show that $\Z_2\times H$ is DR-sequenceable and $K\times H$ is symmetric sequenceable for every $H$ of odd order and every symmetric sequenceable $K$, and deduce that $\Z_{2^k}\times H$ is DR-sequenceable for all $H$ of odd order; in particular all abelian groups except $\Z_2$ are DR-sequenceable (Theorem~\ref{thm:abelianDR}). Along the way we obtain a general doubling principle (Lemma~\ref{lem:doubling}) that converts a sequencing with suitable endpoints into a DR-sequencing, even in groups with no R-sequencing. In Section~\ref{sec:partial} we study extensions of partial sequencings, partial harmonious sequences and partial R-sequencings in large groups, obtaining exact criteria (Lemma~\ref{lem:partial-harm}, Theorem~\ref{thm:partial-seq}), and prove Conjecture~\ref{conj:double} for large groups. Section~\ref{sec:small} contains computational results for small groups.

%=====================================================================
\section{Symmetric harmonious sequences}\label{sec:symharm}
%=====================================================================

Throughout this section $H$ denotes a group of odd order $n$. Then $x\mapsto x^2$ is a bijection of $H$, a fact used repeatedly below. We index sequences of length $n$ by $\Z_n=\{0,1,\dots,n-1\}$.

\begin{definition}
A \emph{symmetric harmonious sequence} in a group $H$ of order $n$ is a harmonious sequence $h_0,h_1,\dots,h_{n-1}$ of $H$ such that
\[
h_ih_{n-i}=1_H\qquad\text{for all }1\le i\le n-1 .
\]
A group with such a sequence is \emph{symmetric harmonious}.
\end{definition}

Symmetric harmonious groups were introduced by Wang and Leonard~\cite{WL94}, who used them to construct R$^*$-sequenceable groups, and they showed that every nilpotent group of odd order is symmetric harmonious~\cite{WL95}. The question is now settled completely.

\begin{theorem}[{\cite[Thm.~1.2]{JavSH}}]\label{thm:symharm}
A finite group is symmetric harmonious if and only if it has odd order. Moreover, in every symmetric harmonious sequence $h_0,\dots,h_{n-1}$ one has $h_0=1_H$.
\end{theorem}

The proof in~\cite{JavSH} lifts symmetric harmonious sequences through normal extensions (if $N\trianglelefteq H$ and both $N$ and $H/N$ are symmetric harmonious, so is $H$), starting from the sequence $0,1,\dots,p-1$ in $\Z_p$; by the Feit--Thompson theorem every group of odd order is solvable, so this suffices. For solvable groups the argument is constructive and avoids Feit--Thompson. In~\cite{JavSH} it is also shown that symmetric harmonious sequences correspond exactly to complete mappings $\varphi$ that are single $|H|$-cycles and satisfy $\iota\varphi\iota=\varphi^{-1}$, where $\iota(x)=x^{-1}$.

Throughout the rest of the paper $H$ denotes a group of odd order $n$, so that $x\mapsto x^2$ is a bijection of $H$, and sequences of length $n$ are indexed by $\Z_n$. If $h$ is a symmetric harmonious sequence and $s_i=h_ih_{i+1}$ (indices mod $n$), then $\{s_0,\dots,s_{n-1}\}=H$, and with $\mu=(n-1)/2$ we have
\begin{equation}\label{eq:sym-s}
s_0=h_1,\qquad s_{n-1}=h_{n-1},\qquad s_\mu=h_\mu h_\mu^{-1}=1_H,\qquad s_{n-1-i}=s_i^{-1}.
\end{equation}

For explicit computations in the abelian case, the following direct construction is convenient; it produces a symmetric harmonious sequence in any direct product of groups of odd order from sequences in the factors, without passing through a composition series.

\begin{example}\label{ex:cyclic}
In $\Z_n$ with $n$ odd, the sequence $h_i=i$ is a symmetric harmonious sequence: $h_i+h_{n-i}=n\equiv 0$, and the sums $h_i+h_{i+1}=2i+1$ ($0\le i\le n-2$) together with $h_{n-1}+h_0=n-1$ are the elements $1,3,\dots,2n-3,\,n-1$, which are distinct modulo $n$.
\end{example}

\begin{lemma}\label{lem:prodsym}
Let $A$ and $B$ be groups of odd orders $p$ and $q$ with symmetric harmonious sequences $a_0,\dots,a_{p-1}$ and $b_0,\dots,b_{q-1}$. Write $\nu=(q-1)/2$ and read indices of the $a_i$ modulo $p$. For $0\le i\le p-1$ and $0\le j\le q-1$ put
\[
c_{iq+j}=\begin{cases}(a_i,\,b_j)& \text{if } 0\le j\le \nu,\\ (a_{i+1},\,b_j) & \text{if } \nu<j\le q-1.\end{cases}
\]
Then $c_0,\dots,c_{pq-1}$ is a symmetric harmonious sequence in $A\times B$.
\end{lemma}

\begin{proof}
For each $j$ the map $i\mapsto c_{iq+j}$ runs through $A\times\{b_j\}$, so $c$ is a permutation of $A\times B$, and $c_0=(a_0,b_0)=1$.

\emph{Symmetry.} Let $1\le t=iq+j\le pq-1$. If $j=0$ then $i\ge 1$, $pq-t=(p-i)q$ and $c_{pq-t}=(a_{p-i},b_0)=(a_i^{-1},1)=c_t^{-1}$. If $1\le j\le \nu$, then $pq-t=(p-1-i)q+(q-j)$ with $q-j>\nu$, so $c_{pq-t}=(a_{p-i},b_{q-j})=(a_i^{-1},b_j^{-1})=c_t^{-1}$. If $\nu<j\le q-1$, then $q-j\le\nu$ and $c_{pq-t}=(a_{p-1-i},b_{q-j})=(a_{i+1}^{-1},b_j^{-1})=c_t^{-1}$.

\emph{Products.} Write $s^B_j=b_jb_{j+1}$. Inside the $i$-th block the consecutive products are $(a_i^2,s^B_j)$ for $0\le j<\nu$, $(a_ia_{i+1},s^B_\nu)=(a_ia_{i+1},1)$ by~\eqref{eq:sym-s}, and $(a_{i+1}^2,s^B_j)$ for $\nu<j\le q-2$. The product from the end of block $i$ to the start of block $i+1$ (cyclically, from block $p-1$ to block $0$) is $(a_{i+1}a_{i+1},\,b_{q-1}b_0)=(a_{i+1}^2,s^B_{q-1})$. As $i$ runs over $\Z_p$, both $a_i$ and $a_{i+1}$ run over $A$, so these products are
\[
\{(a^2,s^B_j):a\in A,\ j\ne\nu\}\ \cup\ \{(a_ia_{i+1},1): i\in\Z_p\}=(A\times(B\setminus\{1\}))\cup(A\times\{1\}),
\]
using that squaring is a bijection of $A$ and that $a$ is harmonious. Hence $c$ is harmonious.
\end{proof}

Combining Example~\ref{ex:cyclic} and Lemma~\ref{lem:prodsym} gives an explicit symmetric harmonious sequence in every abelian group of odd order (the lexicographic pattern of Lemma~\ref{lem:prodsym} also works for non-abelian factors).

%=====================================================================
\section{DR-sequenceable groups}\label{sec:DR}
%=====================================================================

In this section we prove that every group $\Z_{2^k}\times H$ with $H$ of odd order (other than $\Z_2$) is DR-sequenceable; in particular, all abelian groups except $\Z_2$ are DR-sequenceable. For abelian groups, every non-binary one is R-sequenceable~\cite{AKP}, and repeating an R-sequencing twice gives a DR-sequencing. Thus it remains to consider $G=\Z_{2^k}\times H$, where $H$ has odd order and $k\ge 1$; note that $G$ is binary and fails the Hall--Paige condition, so it has no R-sequencing. We first treat $k=1$.

If $H$ is a group of odd order, then $\Z_2\times H$ is symmetric sequenceable~\cite{AI93} (it is solvable by Feit--Thompson) and doubly harmonious when $H$ is abelian~\cite{DJ}, but it is not harmonious, R-sequenceable or $\#$-harmonious, since it fails the Hall--Paige condition. We show that it is nevertheless DR-sequenceable.

To illustrate the idea, let $G=\Z_2\times\Z_5$ and write $(a,b)\in G$ as $\binom{a}{b}$. Start with the symmetric harmonious sequence $0,1,2,3,4$ of $\Z_5$, double it and shuffle in the additive inverses to obtain $1,4,2,3,3,2,4,1,0,\ 1,4,2,3,3,2,4,1,0$. Attaching first coordinates gives the cyclic sequence
\[
\tbinom11,\tbinom14,\tbinom02,\tbinom03,\tbinom13,\tbinom12,\tbinom04,\tbinom01,\tbinom10,\
\tbinom01,\tbinom14,\tbinom12,\tbinom03,\tbinom13,\tbinom02,\tbinom04,\tbinom11,\tbinom10,
\]
whose cyclic consecutive differences are
\[
\tbinom01,\tbinom03,\tbinom13,\tbinom01,\tbinom10,\tbinom04,\tbinom12,\tbinom02,\tbinom14,\
\tbinom11,\tbinom13,\tbinom03,\tbinom11,\tbinom10,\tbinom14,\tbinom02,\tbinom12,\tbinom04 .
\]
Both lists contain every non-zero element of $G$ exactly twice, so this is a DR-sequencing of $\Z_2\times\Z_5$.

\begin{theorem}\label{thm:Z2H}
Let $H$ be a non-trivial group of odd order. Then $G=\Z_2\times H$ is DR-sequenceable.
\end{theorem}

\begin{proof}
By Theorem~\ref{thm:symharm}, $H$ has a symmetric harmonious sequence $h_0=1_H,h_1,\dots,h_{n-1}$, where $n=|H|\ge 3$; let $\mu=(n-1)/2$, and let $s_i=h_ih_{i+1}$ as in~\eqref{eq:sym-s}. For $1\le i\le n-1$ define $\varepsilon_i,\delta_i\in\Z_2$ by
\[
\varepsilon_i\equiv i,\qquad \delta_i\equiv\begin{cases} i+1 & \text{if } i\le\mu,\\ i & \text{if } i>\mu,\end{cases}\pmod 2,
\]
so that $\varepsilon_i+\varepsilon_{n-i}=1$, $\delta_i=\delta_{n-i}$, $\delta_1=\delta_{n-1}=0$, and $\delta_i+\delta_{i+1}=1$ for $i\ne\mu$ while $\delta_\mu+\delta_{\mu+1}=0$. Consider the blocks of length two
\[
P_i\colon (\varepsilon_i,h_i),\,(\varepsilon_i,h_i^{-1}),\qquad Q_i\colon (\delta_i,h_i),\,(1+\delta_i,h_i^{-1})\qquad(1\le i\le n-1),
\]
and the cyclic sequence of length $4n-2$
\[
\bfg\colon\ P_1,P_2,\dots,P_{n-1},\ (1,1_H),\ Q_1,Q_2,\dots,Q_{n-1},\ (1,1_H).
\]

\emph{The terms.} Let $x\in H\setminus\{1_H\}$, say $x=h_j=h_{n-j}^{-1}$ with $1\le j\le n-1$. Among the $P$-blocks, the terms with second coordinate $x$ are $(\varepsilon_j,x)$ and $(\varepsilon_{n-j},x)$, which are $(0,x)$ and $(1,x)$ in some order. Among the $Q$-blocks they are $(\delta_j,x)$ and $(1+\delta_{n-j},x)=(1+\delta_j,x)$, again $(0,x)$ and $(1,x)$. Together with the two occurrences of $(1,1_H)$, every element of $G\setminus\{1_G\}$ occurs exactly twice.

\emph{The quotients.} The cyclic quotients $g_{i-1}^{-1}g_i$ fall into three groups.
\begin{enumerate}
\item[(a)] Inside the blocks: $(0,h_i^{-2})$ inside $P_i$ and $(1,h_i^{-2})$ inside $Q_i$, $1\le i\le n-1$. Since squaring is a bijection of $H$, these are the elements $(0,x)$ and $(1,x)$, $x\ne 1_H$, each once.
\item[(b)] From $P_i$ to $P_{i+1}$ ($1\le i\le n-2$) the quotient is $(\varepsilon_i+\varepsilon_{i+1},h_ih_{i+1})=(1,s_i)$; from $P_{n-1}$, whose last term is $(0,h_{n-1}^{-1})$, to $(1,1_H)$ it is $(1,h_{n-1})=(1,s_{n-1})$; and from $(1,1_H)$ to the first term $(0,h_1)$ of $Q_1$ it is $(1,h_1)=(1,s_0)$. These are the elements $(1,x)$, $x\in H$, each once.
\item[(c)] From $Q_i$ to $Q_{i+1}$ ($1\le i\le n-2$) the quotient is $(1+\delta_i+\delta_{i+1},s_i)$, which equals $(0,s_i)$ for $i\neq\mu$ and $(1,s_\mu)=(1,1_H)$ for $i=\mu$; from $Q_{n-1}$, whose last term is $(1,h_{n-1}^{-1})$, to $(1,1_H)$ it is $(0,s_{n-1})$; and from $(1,1_H)$ to the first term $(1,h_1)$ of $P_1$ it is $(0,s_0)$. These are the elements $(0,x)$, $x\neq 1_H$, each once, together with $(1,1_H)$.
\end{enumerate}
Hence each $(0,x)$ with $x\ne1_H$ occurs once in (a) and once in (c), each $(1,x)$ with $x\ne 1_H$ occurs once in (a) and once in (b), and $(1,1_H)$ occurs once in (b) and once in (c). Thus $\check\bfg\in P_2(G\setminus\{1_G\})$, and $\check\bfg$ is a DR-sequencing of $G$.
\end{proof}

For $\Z_{2^k}\times H$ with $k\ge2$ we use a symmetric sequencing whose second term has a prescribed form. The following theorem proves a more general statement. We first illustrate the construction with $\Z_4\times\Z_3$. The sequence $k\colon 0,1,3,2$ in $\Z_4$ gives the symmetric sequencing $\bar k\colon 0,1,2,3$. Repeat $k$ three times, alternately in regular and reverse order, $0,1,3,2\ |\ 2,3,1,0\ |\ 0,1,3,2$, and pair the half-blocks with the terms of the symmetric harmonious sequence $0,1,2$ of $\Z_3$ (with alternating signs) as follows; here $\binom{a}{b}$ denotes $(b,a)\in\Z_4\times\Z_3$:
\[
\bfg\colon\ \tbinom00,\tbinom01,\tbinom13,\tbinom22,\ \tbinom12,\tbinom23,\tbinom21,\tbinom10,\ \tbinom20,\tbinom11,\tbinom03,\tbinom02 .
\]
The consecutive differences
\[
\bar\bfg\colon\ \tbinom00,\tbinom01,\tbinom12,\tbinom13,\ \tbinom20,\tbinom11,\tbinom02,\tbinom23,\ \tbinom10,\tbinom21,\tbinom22,\tbinom03
\]
form a symmetric sequencing of $\Z_4\times\Z_3$.

\begin{theorem}\label{thm:symseq}
Let $K$ be a symmetric sequenceable group of even order $n$ and let $H$ be a group of odd order $m$. Then $G=K\times H$ is symmetric sequenceable. Moreover, if $k_0=1_K,k_1,\dots,k_{n-1}$ are the partial products of a symmetric sequencing of $K$, then $G$ has a symmetric sequencing whose partial products $b_0,\dots,b_{mn-1}$ satisfy $b_1=(k_1,1_H)$ and $b_{mn-1}=(k_{n-1},1_H)$.
\end{theorem}

\begin{proof}
Let $k_0=1_K,k_1,\dots,k_{n-1}$ be the partial products of a symmetric sequencing of $K$: they form a permutation of $K$, and the quotients $\kappa_r=k_{r-1}^{-1}k_r$ ($1\le r\le n-1$), together with $\kappa_0=1_K$, form a permutation of $K$ with $\kappa_r\kappa_{n-r}=1_K$ for $1\le r\le n-1$. By Theorem~\ref{thm:symharm}, $H$ has a symmetric harmonious sequence $\alpha_0=1_H,\alpha_1,\dots,\alpha_{m-1}$; we read its indices modulo $m$. Write $i=qn+r$ with $0\le q\le m-1$, $0\le r\le n-1$, let $\sigma_r=(-1)^{r+n/2}$, and define
\[
\beta_{q,r}=\begin{cases}\alpha_q^{\sigma_r}& \text{if } r<n/2,\\ \alpha_{q+1}^{\sigma_r} & \text{if } r\ge n/2,\end{cases}
\qquad
g_{qn+r}=\begin{cases}(k_r,\beta_{q,r})&\text{if $q$ is even},\\ (k_{n-1-r},\beta_{q,r})&\text{if $q$ is odd}.\end{cases}
\]
We show that $g_0,\dots,g_{mn-1}$ are the partial products of a symmetric sequencing of $G$. Clearly $g_0=(1_K,1_H)$.

\emph{$\bfg$ is a permutation of $G$.} Fix $0\le j\le n-1$ and let $\epsilon=\sigma_j$. The terms with first coordinate $k_j$ occur at $r=j$ in even blocks and at $r=n-1-j$ in odd blocks, and $\sigma_{n-1-j}=-\epsilon$. If $j<n/2$, their second coordinates are $\alpha_q^{\epsilon}$ ($q$ even) and $\alpha_{q+1}^{-\epsilon}=\alpha_{m-1-q}^{\epsilon}$ ($q$ odd). Since $m$ is odd, $m-1-q$ is odd when $q$ is odd, so these are $\alpha_t^{\epsilon}$ for $t$ even and for $t$ odd respectively, i.e.\ all of $H$. The case $j\ge n/2$ is identical.

\emph{The quotients.} Let $\bg_i=g_{i-1}^{-1}g_i$ for $1\le i\le mn-1$ and $\bg_0=g_0$. In the first coordinate, for $r\ge1$ we get $\kappa_r$ in even blocks and $k_{n-r}^{-1}k_{n-1-r}=\kappa_{n-r}^{-1}=\kappa_r$ in odd blocks; at $r=0$, $q\ge 1$, consecutive blocks meet at $k_{n-1}$ or at $k_0$, giving $1_K=\kappa_0$. In the second coordinate, since $\sigma_{r-1}=-\sigma_r$, a direct computation gives
\begin{equation}\label{eq:quot}
\bg_{qn+r}=\begin{cases}
(1_K,\ \alpha_q^{2\sigma_0}) & r=0,\\
(\kappa_r,\ \alpha_q^{2\sigma_r}) & 0<r<n/2,\\
(\kappa_{n/2},\ \alpha_q\alpha_{q+1}) & r=n/2,\\
(\kappa_r,\ \alpha_{q+1}^{2\sigma_r}) & n/2<r\le n-1,
\end{cases}
\end{equation}
where for $r=0$ we used that the last term of block $q-1$ has second coordinate $\alpha_q^{\sigma_{n-1}}=\alpha_q^{-\sigma_0}$, and for $r=n/2$ that $\sigma_{n/2-1}=-1$, $\sigma_{n/2}=1$. (For $(q,r)=(0,0)$ formula~\eqref{eq:quot} gives $(1_K,1_H)=\bg_0$.) For each fixed $r$, the second coordinates in~\eqref{eq:quot} run over $H$ as $q$ runs over $\Z_m$: for $r\ne n/2$ because $x\mapsto x^{\pm 2}$ is a bijection, and for $r=n/2$ because $\alpha$ is harmonious. Hence $\bar\bfg$ is a permutation of $G$.

\emph{Symmetry.} Let $1\le i=qn+r\le mn-1$. If $r=0$, then $mn-i=(m-q)n$, and $\bg_{mn-i}=(1_K,\alpha_{m-q}^{2\sigma_0})=(1_K,\alpha_q^{-2\sigma_0})=\bg_i^{-1}$. If $0<r<n/2$, then $mn-i=(m-1-q)n+(n-r)$ with $n-r>n/2$ and $\sigma_{n-r}=\sigma_r$, so $\bg_{mn-i}=(\kappa_{n-r},\alpha_{m-q}^{2\sigma_r})=(\kappa_r^{-1},\alpha_q^{-2\sigma_r})=\bg_i^{-1}$; the case $n/2<r<n$ is symmetric. If $r=n/2$, then $mn-i=(m-1-q)n+n/2$ and $\bg_{mn-i}=(\kappa_{n/2},\alpha_{m-1-q}\alpha_{m-q})=(\kappa_{n/2}^{-1},\alpha_{q+1}^{-1}\alpha_q^{-1})=\bg_i^{-1}$, using $\kappa_{n/2}^2=1_K$. Thus $\bg_i\bg_{mn-i}=1_G$ for all $1\le i\le mn-1$. Finally, $b_1=g_1=(k_1,\alpha_0^{\sigma_1})=(k_1,1_H)$ and $b_{mn-1}=g_{mn-1}=(k_{n-1},\alpha_m^{\sigma_{n-1}})=(k_{n-1},1_H)$, since $m-1$ is even.
\end{proof}

\begin{remark}
If $H$ is abelian, the sign $\sigma_r$ may be replaced by $(-1)^r$. For non-abelian $H$ the normalization $\sigma_{n/2}=1$ is needed when $n\equiv 2\pmod 4$: with $\sigma_r=(-1)^r$ the middle quotients become $\alpha_q^{-1}\alpha_{q+1}^{-1}=(\alpha_{q+1}\alpha_q)^{-1}$, which need not run over $H$. (For instance, the construction with $\sigma_r=(-1)^r$ fails for $K=\Z_2$ or $\Z_6$ and $H=\Z_7\rtimes\Z_3$.)
\end{remark}

\begin{corollary}
If $K$ is a symmetric sequenceable group and $H$ is any group of odd order, then $K\times H$ is symmetric sequenceable. In particular, if Anderson's conjecture holds for a binary group $K$, it holds for $K\times H$.
\end{corollary}

Next we isolate the doubling argument that handles $k\ge2$. It applies to any sequenceable group.

\begin{lemma}[Doubling]\label{lem:doubling}
Let $G$ be a group of order $N$ with a sequencing whose partial products are $b_0=1_G,b_1,\dots,b_{N-1}$. Suppose there is an automorphism $\varphi$ of $G$ with
\[
\varphi(b_1)=b_{N-1}b_1\qquad\text{and}\qquad \varphi(b_{N-1}b_1)=b_1 .
\]
Then the cyclic sequence $b_1,\dots,b_{N-1},\varphi(b_1),\dots,\varphi(b_{N-1})$ is a DR-sequencing of $G$. In particular, $G$ is DR-sequenceable.
\end{lemma}

\begin{proof}
Since $b$ is a permutation of $G$ with $b_0=1_G$ and $\varphi$ fixes $1_G$, every element of $G\setminus\{1_G\}$ occurs exactly twice. Let $c_i=b_{i-1}^{-1}b_i$, $1\le i\le N-1$; these form a permutation of $G\setminus\{1_G\}$, and $c_1=b_1$. The cyclic quotients of the doubled sequence are $c_2,\dots,c_{N-1}$, $\varphi(c_2),\dots,\varphi(c_{N-1})$ and the two quotients across the joins:
\[
b_{N-1}^{-1}\varphi(b_1)=b_1=c_1,\qquad \varphi(b_{N-1})^{-1}b_1=\varphi(b_{N-1})^{-1}\varphi(b_{N-1}b_1)=\varphi(b_1)=\varphi(c_1).
\]
Hence the quotients are $c_1,\dots,c_{N-1},\varphi(c_1),\dots,\varphi(c_{N-1})$, which contain each element of $G\setminus\{1_G\}$ exactly twice.
\end{proof}

\begin{corollary}\label{cor:binary-doubling}
Let $G$ be a binary group with unique involution $z$, and let $\chi\colon G\to\Z_2$ be a homomorphism with $\chi(z)=0$. Suppose $G$ has a sequencing whose partial products $b_0=1_G,b_1,\dots,b_{N-1}$ satisfy $b_{N-1}=z$ and $\chi(b_1)=1$. Then $G$ is DR-sequenceable. The condition $b_{N-1}=z$ holds for every symmetric sequencing, and for every sequencing if $G$ is abelian.
\end{corollary}

\begin{proof}
Since $z$ is central, $\varphi(x)=xz^{\chi(x)}$ is an endomorphism of $G$, and $\varphi^2=\mathrm{id}$ because $\chi(z)=0$; so $\varphi$ is an automorphism. Now $\varphi(b_1)=zb_1=b_{N-1}b_1$ and $\varphi(zb_1)=zb_1z=b_1$, so Lemma~\ref{lem:doubling} applies. If $\bar\bfg$ is a symmetric sequencing of $G$ of order $N$, then $\bg_{N/2+1}^2=1_G\neq\bg_{N/2+1}$, so $\bg_{N/2+1}=z$, and
\[
b_{N-1}=\bg_2\cdots\bg_N=(\bg_2\cdots\bg_{N/2})\,z\,(\bg_2\cdots\bg_{N/2})^{-1}=z .
\]
If $G$ is abelian, then $b_{N-1}=\Sig G=z$ by Lemma~\ref{lem:sigma}(b).
\end{proof}

A homomorphism $\chi$ as in Corollary~\ref{cor:binary-doubling} exists, for instance, whenever the Sylow $2$-subgroup of the binary group $G$ is cyclic of order at least $4$: by Burnside's transfer theorem $G$ has a normal $2$-complement $M$ with $G/M$ cyclic of order $2^k\ge4$, and $z\in\ker\chi$ for the surjection $\chi\colon G\to G/M\to\Z_2$. It also exists for every dicyclic group $\Dic_{4m}$, $m\geq 2$, namely $\chi(a)=0$, $\chi(b)=1$.

We are now ready to prove the main theorem of this section.

\begin{theorem}\label{thm:abelianDR}
Let $H$ be a group of odd order and $k\geq 1$, with $(k,|H|)\neq(1,1)$. Then $\Z_{2^k}\times H$ is DR-sequenceable. In particular, every abelian group except $\Z_2$ is DR-sequenceable.
\end{theorem}

\begin{proof}
Let $G=\Z_{2^k}\times H$ with $|H|=m$. If $k=1$, then $H$ is non-trivial and Theorem~\ref{thm:Z2H} applies. (For the last assertion, a non-binary abelian group is R-sequenceable~\cite{AKP}, hence DR-sequenceable, and a binary abelian group has the form $\Z_{2^k}\times H$.)

Suppose $k\ge2$ and let $n=2^k$. In $\Z_n$ the sequence $0,1,-1,2,-2,\dots,n/2-1,-(n/2-1),n/2$ has consecutive differences $0,1,-2,3,-4,\dots,n-1$, which form a symmetric sequencing of $\Z_n$. Its partial products are $k_0,\dots,k_{n-1}$ with $k_1=1$ and $k_{n-1}=n/2$. By Theorem~\ref{thm:symseq} there is a symmetric sequencing of $G$ whose partial products $b_0,\dots,b_{mn-1}$ satisfy $b_1=(1,1_H)$ and $b_{mn-1}=(n/2,1_H)=z$, the unique involution of $G$. With the homomorphism $\chi(x,h)=x\bmod 2$ we have $\chi(z)=0$ (as $k\ge2$) and $\chi(b_1)=1$, so Corollary~\ref{cor:binary-doubling} applies.
\end{proof}

\begin{remark}
If $H$ is abelian and $k\geq 2$, the automorphism $\varphi(x)=x+\chi(x)z$ is multiplication by $1+mn/2$, and the DR-sequencing of Lemma~\ref{lem:doubling} is $b_1,\dots,b_{mn-1},(1+\tfrac{mn}{2})b_1,\dots,(1+\tfrac{mn}{2})b_{mn-1}$. For $k=1$ no such $\chi$ exists, which is why Theorem~\ref{thm:Z2H} requires a different construction.
\end{remark}

The same argument gives, more generally: \emph{if $K$ has a symmetric sequencing with partial products $k_0,\dots,k_{n-1}$ and $\chi\colon K\to\Z_2$ is a homomorphism with $\chi(k_{n-1})=0$ and $\chi(k_1)=1$, then $K\times H$ is DR-sequenceable for every group $H$ of odd order} (combine Theorem~\ref{thm:symseq} and Corollary~\ref{cor:binary-doubling}, extending $\chi$ trivially to $H$). Corollary~\ref{cor:binary-doubling} also applies to binary groups with cyclic Sylow $2$-subgroups that are not direct products. For example, for odd $m$ the dicyclic group $\Dic_{4m}$ has Sylow $2$-subgroup $\Z_4$, and $\chi(a)=0$, $\chi(b)=1$ defines a homomorphism with $\chi(a^m)=0$.

\begin{proposition}\label{prop:dic}
For every odd $m$ with $3\le m\le 13$, the dicyclic group $\Dic_{4m}$ has a sequencing with $b_1=b$ and $b_{4m-1}=a^m$ (listed in Table~\ref{tab:dic}). Consequently $\Dic_{4m}$ is DR-sequenceable for these $m$, although it is not R-sequenceable.
\end{proposition}

\begin{question}
Is $\Dic_{4m}$ DR-sequenceable for every odd $m\geq 3$? More generally, does every binary group with cyclic Sylow $2$-subgroup of order at least $4$ have a sequencing satisfying the hypotheses of Corollary~\ref{cor:binary-doubling}?
\end{question}

%=====================================================================
\section{Partial sequences and double sequences in large groups}\label{sec:partial}
%=====================================================================

M\"uyesser and Pokrovskiy~\cite{MP} used techniques from probabilistic combinatorics to prove that all sufficiently large non-abelian groups are sequenceable. To apply their results we use the following definition, based on their notation.

\begin{definition}\label{def:rainbow}
Let $G$ be a group and $A,B\subseteq G$ with $|A|=m$ and $|B|=m-1$. A \emph{rainbow Hamilton sequence} in $K^+[A;B]$ is a sequence $g_1,\dots,g_m$ such that $A=\{g_1,\dots,g_m\}$ and $B=\{g_1g_2,\dots,g_{m-1}g_m\}$. A rainbow Hamilton sequence in $K^-[A;B]$ is a sequence $g_1,\dots,g_m$ such that $A=\{g_1,\dots,g_m\}$ and $B=\{g_1^{-1}g_2,\dots,g_{m-1}^{-1}g_m\}$. In both cases the sequence is said to be \emph{from $g_1$ to $g_m$}.
\end{definition}

The results of this section are consequences of the following theorem.

\begin{theorem}[{\cite[Thm.~6.9]{MP}}]\label{thm:MP}
Let $G$ be a sufficiently large group of order $n$. Let $V,C\subseteq G$ and $x,y\in G$ be such that $|V|+1=|C|\ge n-\sqrt n$, $x\ne y$, $x,y\notin V$, and suppose further that $1_G\notin C$ if $G$ is an elementary abelian $2$-group.
\begin{enumerate}
\item[(i)] If $\Sig C=y-x$, then there exists a rainbow Hamilton sequence from $x$ to $y$ in $K^-[\{x,y\}\cup V;C]$.
\item[(ii)] If $\Sig C=x+y+2\Sig V$, then there exists a rainbow Hamilton sequence from $x$ to $y$ in $K^+[\{x,y\}\cup V;C]$.
\end{enumerate}
\end{theorem}

Note that the conditions on $\Sig C$ in Theorem~\ref{thm:MP} are also necessary, since $\sum_{i}(\pi(g_{i+1})-\pi(g_i))=y-x$ and $\sum_i(\pi(g_i)+\pi(g_{i+1}))=x+y+2\Sig V$ for any sequence $x=g_1,\dots,g_m=y$ running through $\{x,y\}\cup V$.

\subsection{Partial harmonious sequences and partial R-sequencings}

\begin{definition}
Let $W\subseteq G$.
\begin{enumerate}
\item[(a)] A sequence $w_1,\dots,w_k$ of distinct elements of $W$ is a \emph{partial harmonious sequence in $W$} if the linear products $w_1w_2,\dots,w_{k-1}w_k$ are distinct elements of $W$. It is \emph{extendable in $W$} if there are $w_{k+1},\dots,w_m$ such that $w_1,\dots,w_m$ is a harmonious sequence of $W$.
\item[(b)] Let $W^*=W\setminus\{1_G\}$. A sequence $w_1,\dots,w_k$ of distinct elements of $W^*$ is a \emph{partial R-sequencing in $W$} if the linear quotients $w_1^{-1}w_2,\dots,w_{k-1}^{-1}w_k$ are distinct elements of $W^*$. It is \emph{extendable in $W$} if there are $w_{k+1},\dots,w_m$ such that $\bfg\colon w_1,\dots,w_m$ belongs to $P_1(W^*)$ and $\check\bfg\colon w_m^{-1}w_1,w_1^{-1}w_2,\dots,w_{m-1}^{-1}w_m$ is an R-sequencing of $W$.
\end{enumerate}
\end{definition}

When can a partial harmonious sequence or a partial R-sequencing be extended? One expects this to be possible when the partial sequence is short compared with $|G|$ and $G$ satisfies the Hall--Paige condition. We first show that a length of order $|G|/4$ does not suffice, even for groups satisfying the Hall--Paige condition.

\begin{example}\label{ex:nonext}
Let $H$ be a group of even order $n$ and let $G=H\times\Z_4$, a group of order $4n$ that satisfies the Hall--Paige condition. Call a maximal cyclic run of consecutive terms of a cyclic sequence in $G$ whose second coordinates have the same parity a \emph{block}, and call it \emph{even} or \emph{odd} accordingly.

\emph{Harmonious sequences.} Let $g_1,\dots,g_{4n}$ be a harmonious sequence of $G$. A product $g_{i-1}g_i$ has odd second coordinate exactly when $g_{i-1}$ and $g_i$ lie in blocks of different parity, i.e.\ at a block boundary. Since exactly $2n$ elements of $G$ have odd second coordinate, there are exactly $2n$ blocks, alternately even and odd. The $n$ even blocks are non-empty and together contain the $2n$ elements with even second coordinate, so each even block has at most $n+1$ terms. Now suppose $H$ is harmonious, with harmonious sequence $h_1,\dots,h_n$. Then
\[
(h_1,0),(h_2,0),\dots,(h_n,0),\ (h_1,2),\ (h_1^{-1}h_nh_1,2)
\]
is a partial harmonious sequence in $G$ of length $n+2$: its terms are distinct, and its linear products $(h_ih_{i+1},0)$ for $1\le i\le n-1$, $(h_nh_1,2)$, and $(h_nh_1,0)$ are distinct because $h_nh_1\notin\{h_ih_{i+1}:1\le i\le n-1\}$. Its terms all have even second coordinate, so it would lie in a single even block of length at least $n+2$. Hence it cannot be extended to a harmonious sequence of $G$.

\emph{R-sequencings.} The same argument applies to an R-sequencing of $G$: there are again $2n$ blocks, and the $n$ even blocks contain the $2n-1$ non-identity elements with even second coordinate, so each has at most $n$ terms. If $h_1,\dots,h_{n-1}$ is a sequence of $H\setminus\{1_H\}$ whose cyclic quotients form an R-sequencing of $H$, then
\[
(h_1,0),\dots,(h_{n-1},0),\ (1_H,2),\ (h_{n-1}^{-1}h_1,2)
\]
is a partial R-sequencing in $G$ of length $n+1$ (its linear quotients are $(h_i^{-1}h_{i+1},0)$ for $1\le i\le n-2$, $(h_{n-1}^{-1},2)$ and $(h_{n-1}^{-1}h_1,0)$), and it cannot be extended.
\end{example}

The next lemma shows that, in contrast, partial harmonious sequences and partial R-sequencings of length up to about $\sqrt{|G|}$ are extendable in a sufficiently large group whenever the obvious necessary condition holds.

\begin{lemma}\label{lem:partial-harm}
Let $G$ be a sufficiently large group of order $n$, let $W\subseteq G$ with $\Sig W=0$, and let $W^*=W\setminus\{1_G\}$.
\begin{enumerate}
\item[(a)] Suppose that $1_G\notin W$ if $G$ is an elementary abelian $2$-group. Then every partial harmonious sequence in $W$ of length $k$ with $2\le k\le |W|-n+\sqrt n+1$ is extendable in $W$.
\item[(b)] Every partial R-sequencing in $W$ of length $k$ with $2\le k\le |W^*|-n+\sqrt n+1$ is extendable in $W$.
\end{enumerate}
Conversely, if $\Sig W\ne 0$, then $W$ is neither harmonious nor R-sequenceable, so no partial sequence in $W$ is extendable.
\end{lemma}

\begin{proof}
(a) Let $w_1,\dots,w_k$ be a partial harmonious sequence in $W$. Put $x=w_k$, $y=w_1$, $V=W\setminus\{w_1,\dots,w_k\}$ and $C=W\setminus\{w_1w_2,\dots,w_{k-1}w_k\}$. Then $x\neq y$, $|V|+1=|C|=|W|-k+1\ge n-\sqrt n$, and, using $\Sig W=0$,
\[
\Sig C=\Sig W-\sum_{i=1}^{k-1}(w_i+w_{i+1})=\Sig W+w_1+w_k-2\Big(\Sig W-\Sig V\Big)=x+y+2\Sig V .
\]
By Theorem~\ref{thm:MP}(ii) there is a rainbow Hamilton sequence $w_k,w_{k+1},\dots,w_m,w_{m+1}=w_1$ from $x$ to $y$ in $K^+[\{x,y\}\cup V;C]$. Its consecutive products are exactly the elements of $C$, so $w_1,\dots,w_m$ is a harmonious sequence of $W$ extending the given one.

(b) Let $w_1,\dots,w_k$ be a partial R-sequencing in $W$. Put $x=w_k$, $y=w_1$, $V=W^*\setminus\{w_1,\dots,w_k\}$ and $C=W^*\setminus\{w_1^{-1}w_2,\dots,w_{k-1}^{-1}w_k\}$. Then $1_G\notin C$, $|V|+1=|C|=|W^*|-k+1\ge n-\sqrt n$, and
\[
\Sig C=\Sig W^*-\sum_{i=1}^{k-1}(w_{i+1}-w_i)=\Sig W-w_k+w_1=y-x .
\]
By Theorem~\ref{thm:MP}(i) there is a rainbow Hamilton sequence $w_k,\dots,w_m,w_1$ in $K^-[\{x,y\}\cup V;C]$, and $w_1,\dots,w_m$ is the required extension.

The converse is Lemma~\ref{lem:sigma}(c).
\end{proof}

\begin{remark}
Lemma~\ref{lem:partial-harm} slightly improves the range $k\le|W|-n+\sqrt{n}$ that one gets by not distinguishing $W$ from $W^*$, and Example~\ref{ex:nonext} shows that the range cannot be extended to $k\approx |G|/4$. It would be interesting to determine the largest $f(n)$ such that every partial harmonious sequence of length at most $f(|G|)$ in a large group $G$ satisfying the Hall--Paige condition is extendable; we have $\sqrt n\lesssim f(n)\le n/4+1$.
\end{remark}

\begin{corollary}\label{cor:harm-char}
Let $G$ be a sufficiently large group of order $n$ and let $W\subseteq G$, $W^*=W\setminus\{1_G\}$.
\begin{enumerate}
\item[(a)] If $|W|\ge n-\sqrt n+1$, then $W$ is harmonious if and only if $\Sig W=0$ and, in case $G$ is an elementary abelian $2$-group, $1_G\notin W$.
\item[(b)] If $|W^*|\ge n-\sqrt n+1$, then $W$ is R-sequenceable if and only if $\Sig W=0$.
\end{enumerate}
\end{corollary}

\begin{proof}
Necessity follows from Lemma~\ref{lem:sigma}(c); moreover, in an elementary abelian $2$-group a product of two distinct elements is never $1_G$, so $1_G\notin W$ is necessary in (a).
For sufficiency in (a), fix $w_1\in W$. Since $|W\cap w_1^{-1}W|\ge 2|W|-n\geq 2$, there is $w_2\in W\setminus\{w_1\}$ with $w_1w_2\in W$, and $w_1,w_2$ is a partial harmonious sequence of length $2\le|W|-n+\sqrt n+1$. In (b), fix $w_1\in W^*$; since $|W^*\cap w_1W^*|\ge 2|W^*|-n>0$, there is $w_2\in W^*$ with $w_1^{-1}w_2\in W^*$, and then $w_2\neq w_1$. In both cases Lemma~\ref{lem:partial-harm} completes the proof.
\end{proof}

The following corollary extends the corresponding abelian results~\cite{BGHJ,DJ} to all large groups.

\begin{corollary}\label{cor:HPminus}
Let $G$ be a sufficiently large group of order $n$ and let $X\subseteq G$.
\begin{enumerate}
\item[(a)] If $|X|\le\sqrt n-1$, then $G\setminus X$ is harmonious if and only if $\Sig X=\Sig G$ (and, if $G$ is an elementary abelian $2$-group, $1_G\in X$).
\item[(b)] If $1_G\in X$ and $|X|\le\sqrt n-1$, then $G\setminus X$ is R-sequenceable if and only if $\Sig X=\Sig G$.
\end{enumerate}
In particular:
\begin{enumerate}
\item[(i)] if $G$ satisfies the Hall--Paige condition, then $G\setminus\{1_G\}$ is harmonious, i.e.\ $G$ is $\#$-harmonious;
\item[(ii)] if $G$ is binary with unique involution $\iota_G$, then $G\setminus\{\iota_G\}$ and $G\setminus\{1_G,\iota_G\}$ are harmonious;
\item[(iii)] for every $x\in G$ with $\pi(x)=\Sig G$, the set $G\setminus\{x\}$ is harmonious; for example, if $n\equiv 2\pmod 4$, then $G$ minus any element outside the subgroup of index $2$ is harmonious.
\end{enumerate}
\end{corollary}

\begin{proof}
Apply Corollary~\ref{cor:harm-char} to $W=G\setminus X$, noting $\Sig W=\Sig G-\Sig X$. Parts (i) and (ii) follow from Lemma~\ref{lem:sigma}(b). For (iii) with $n\equiv 2\pmod 4$, a Sylow $2$-subgroup has order $2$, $G$ has a normal $2$-complement $M$ of index $2$, and $\Sig G$ is the non-zero element of the $2$-part $G/M$ of $G^{\mathrm{ab}}$ (see the proof of Lemma~\ref{lem:sigma}); hence $\pi(x)=\Sig G$ for every $x\notin M$.
\end{proof}

\subsection{Partial sequencings}

A \emph{partial sequencing} of length $k$ in $G$ is a sequence $g_0=1_G,g_1,\dots,g_{k-1}$ of distinct elements whose partial products $a_i=g_0g_1\cdots g_i$, $0\le i\le k-1$, are distinct. It \emph{extends} to a sequencing of $G$ if there are $g_k,\dots,g_{n-1}$ such that $g_0,\dots,g_{n-1}$ is a sequencing of $G$.

\begin{theorem}\label{thm:partial-seq}
Let $G$ be a sufficiently large group of order $n$, and let $g_0=1_G,g_1,\dots,g_{k-1}$ be a partial sequencing in $G$ with partial products $a_0,\dots,a_{k-1}$, where $1\le k\le\sqrt n$. Then it extends to a sequencing of $G$ if and only if there exists
\[
y\in G\setminus\{a_0,\dots,a_{k-1}\}\qquad\text{with}\qquad \pi(y)=\Sig G .
\]
\end{theorem}

\begin{proof}
If $g_0,\dots,g_{n-1}$ is a sequencing extending the given partial sequencing, then $y=a_{n-1}=g_0g_1\cdots g_{n-1}$ satisfies $\pi(y)=\Sig G$, and $y\notin\{a_0,\dots,a_{k-1}\}$ because $k<n$.

Conversely, let $y$ be as in the statement. Put $x=a_{k-1}$, $V=G\setminus\{a_0,\dots,a_{k-1},y\}$ and $C=G\setminus\{g_0,g_1,\dots,g_{k-1}\}$. Then $x\ne y$, $x,y\notin V$, $1_G=g_0\notin C$, and $|V|+1=n-k=|C|\ge n-\sqrt n$. Since $g_1\cdots g_{k-1}=a_{k-1}$,
\[
\Sig C=\Sig G-\pi(a_{k-1})=\pi(y)-\pi(x)=y-x .
\]
By Theorem~\ref{thm:MP}(i) there is a rainbow Hamilton sequence $x=a_{k-1},a_k,\dots,a_{n-1}=y$ in $K^-[\{x,y\}\cup V;C]$. Setting $g_i=a_{i-1}^{-1}a_i$ for $k\le i\le n-1$, the elements $g_k,\dots,g_{n-1}$ are exactly the elements of $C$, and $a_0,\dots,a_{n-1}$ is a permutation of $G$. Hence $g_0,\dots,g_{n-1}$ is a sequencing of $G$ extending the given one.
\end{proof}

The set $\pi^{-1}(\Sig G)$ is a coset of $[G,G]$, so the condition of Theorem~\ref{thm:partial-seq} can fail only if $|[G,G]|\le k$.

\begin{corollary}\label{cor:partial-seq}
Let $G$ be a sufficiently large group of order $n$.
\begin{enumerate}
\item[(a)] If $k\le\sqrt n$ and $|[G,G]|>k$, then every partial sequencing of length $k$ extends to a sequencing of $G$. This applies to every $k\le\sqrt n$ if $|[G,G]|>\sqrt n$ (for example, if $G$ is perfect).
\item[(b)] If $G$ is abelian and binary with involution $\iota_G$, then a partial sequencing of length $k\le\sqrt n$ extends to a sequencing if and only if $\iota_G$ is not one of its partial products.
\item[(c)] (Case $k=1$.) $G$ is sequenceable if and only if $G$ is non-abelian or binary abelian.
\end{enumerate}
\end{corollary}

\begin{proof}
(a) is immediate. For (b), $[G,G]$ is trivial and $\Sig G=\iota_G$. For (c), take the partial sequencing $g_0=1_G$: an admissible $y\ne 1_G$ exists unless $[G,G]=1$ and $\Sig G=0$, i.e.\ unless $G$ is abelian and not binary (Lemma~\ref{lem:sigma}(b)).
\end{proof}

Part (c) recovers, for large groups, the theorems of Gordon~\cite{Gordon} and M\"uyesser--Pokrovskiy~\cite{MP}. The hypothesis in Theorem~\ref{thm:partial-seq} cannot be dropped, even for non-abelian groups, as the following example shows.

\begin{example}\label{ex:S3}
Let $m\ge 5$ be odd and $G=S_3\times\Z_m$, of order $n=6m$. Then $[G,G]=A_3\times\{0\}$, $G^{\mathrm{ab}}\cong\Z_2\times\Z_m$ via $\pi(\sigma,t)=(\operatorname{sgn}\sigma,t)$, and
\[
\Sig G=\Big(m\sum_{\sigma\in S_3}\operatorname{sgn}\sigma,\ 6\sum_{t\in\Z_m}t\Big)=(3m,\,3m(m-1))=(1,0),
\]
so $\pi^{-1}(\Sig G)=\{(\tau,0):\tau\in S_3\text{ a transposition}\}$. Let $\tau_1,\tau_2,\tau_3$ be the three transpositions and consider
\[
g_0=(1,0),\quad g_1=(\tau_1,0),\quad g_2=(\tau_1,1),\quad g_3=(\tau_2,-1),\quad g_4=(\tau_2\tau_3,0).
\]
These are distinct, and the partial products are $(1,0),\ (\tau_1,0),\ (1,1),\ (\tau_2,0),\ (\tau_3,0)$, which are distinct. So this is a partial sequencing of length $5\le\sqrt{6m}$, and its partial products contain $\pi^{-1}(\Sig G)$. By Theorem~\ref{thm:partial-seq} (or directly by the necessity argument in its proof), it does not extend to a sequencing of $G$, although $G$ is sequenceable for large $m$~\cite{MP}.
\end{example}

\subsection{Double sequences}

\begin{lemma}\label{lem:single-to-double}
Let $G$ be a finite group and $A\subseteq G$.
\begin{enumerate}
\item[(a)] If $G$ is sequenceable, then $G$ is doubly sequenceable.
\item[(b)] If $A$ is harmonious (resp.\ R-sequenceable), then $A$ is doubly harmonious (resp.\ doubly R-sequenceable).
\end{enumerate}
\end{lemma}

\begin{proof}
(a) Let $b_1=1_G,b_2,\dots,b_n$ be the partial products of a sequencing, so that the linear quotients $1_G,b_1^{-1}b_2,\dots,b_{n-1}^{-1}b_n$ form a permutation of $G$. Then the sequence $b_1,\dots,b_n,b_nb_1,\dots,b_nb_n$ lies in $P_2(G)$, starts with $1_G$, and its linear quotients are the same list twice, since $b_n^{-1}(b_nb_1)=1_G$ and $(b_nb_{i-1})^{-1}b_nb_i=b_{i-1}^{-1}b_i$. (b) Repeat a harmonious sequence (resp.\ an R-sequencing) twice.
\end{proof}

The next lemma gives necessary and sufficient conditions for a large subset of a sufficiently large group to be doubly harmonious or doubly R-sequenceable.

\begin{lemma}\label{lem:double}
Let $G$ be a sufficiently large group of order $n$ that is not an elementary abelian $2$-group.
\begin{enumerate}
\item[(a)] If $W\subseteq G$ and $|W|\ge n-\sqrt{n/2}+1$, then $W$ is doubly harmonious if and only if $2\Sig W=0$.
\item[(b)] If $W\subseteq G\setminus\{1_G\}$ and $|W|\ge n-\sqrt{n/2}+1$, then $W$ is doubly R-sequenceable if and only if $2\Sig W=0$.
\end{enumerate}
\end{lemma}

\begin{proof}
Necessity is Lemma~\ref{lem:sigma}(c). For sufficiency let $\Gamma=G\times\Z_2$, a sufficiently large group of order $2n$ that is not an elementary abelian $2$-group, with $\Gamma^{\mathrm{ab}}=G^{\mathrm{ab}}\times\Z_2$, and let $\tau\colon\Gamma\to G$ be the projection onto the first coordinate. Let $t\in\Z_2$ be the parity of $|W|$, so that $\Sig(W\times\Z_2)=(2\Sig W,t)=(0,t)$. In both cases $2|W|-1\ge 2n-\sqrt{2n}$, so the size condition of Theorem~\ref{thm:MP} holds for subsets of $\Gamma$ of size $2|W|-1$.

(a) Fix $x\in W$. Since $|W\cap W^{-1}x|\ge 2|W|-n>0$, there are $v_1,v_2\in W$ (not necessarily distinct) with $v_2v_1=x$. Let
\[
V=(W\times\Z_2)\setminus\{(v_1,0),(v_2,1)\},\qquad C=(W\times\Z_2)\setminus\{(x,1-t)\}.
\]
Then $|V|+1=|C|=2|W|-1$, and in $\Gamma^{\mathrm{ab}}$
\[
\begin{aligned}
(v_1,0)+(v_2,1)+2\Sig V&=(v_1+v_2,1)+2\big((0,t)-(v_1+v_2,1)\big)=(-x,1),\\
\Sig C&=(0,t)-(x,1-t)=(-x,1).
\end{aligned}
\]
By Theorem~\ref{thm:MP}(ii) there is a rainbow Hamilton sequence $\gamma_1,\dots,\gamma_{2|W|}$ in $K^+[W\times\Z_2;C]$ from $\gamma_1=(v_1,0)$ to $\gamma_{2|W|}=(v_2,1)$. The first coordinates $\tau(\gamma_1),\dots,\tau(\gamma_{2|W|})$ contain each element of $W$ exactly twice. Among the products $\tau(\gamma_{i-1}\gamma_i)$, $2\le i\le 2|W|$, each element of $W\setminus\{x\}$ occurs twice and $x$ occurs once, and the cyclic product $\tau(\gamma_{2|W|}\gamma_1)=v_2v_1=x$ supplies the second occurrence of $x$. Thus $\tau(\gamma_1),\dots,\tau(\gamma_{2|W|})$ is a doubly harmonious sequence of $W$.

(b) Fix $y\in W$. Since $|W\cap Wy|\ge 2|W|-n>0$, there are $u_1,u_2\in W$ with $u_1=u_2y$, i.e.\ $u_2^{-1}u_1=y$; as $y\ne 1_G$, $u_1\ne u_2$. Let
\[
V'=(W\times\Z_2)\setminus\{(u_1,0),(u_2,1)\},\qquad C'=(W\times\Z_2)\setminus\{(y,1-t)\}.
\]
Then $1_\Gamma\notin C'$, $|V'|+1=|C'|=2|W|-1$, and $\Sig C'=(0,t)-(y,1-t)=(-y,1)=(u_2,1)-(u_1,0)$. By Theorem~\ref{thm:MP}(i) there is a rainbow Hamilton sequence $\gamma_1,\dots,\gamma_{2|W|}$ in $K^-[W\times\Z_2;C']$ from $(u_1,0)$ to $(u_2,1)$. As in (a), its first coordinates contain each element of $W$ twice, the linear quotients contribute each element of $W\setminus\{y\}$ twice and $y$ once, and the cyclic quotient $\tau(\gamma_{2|W|}^{-1}\gamma_1)=u_2^{-1}u_1=y$ supplies the second occurrence of $y$. This is a double R-sequencing of $W$.
\end{proof}

\begin{theorem}\label{thm:large-double}
Every sufficiently large group $G$ is doubly sequenceable, doubly R-sequenceable, doubly harmonious, and doubly $\#$-harmonious.
\end{theorem}

\begin{proof}
For abelian $G$ these properties are proved in~\cite{Jav24}, Theorem~\ref{thm:abelianDR}, and~\cite{DJ}. Let $G$ be non-abelian of order $n$. It is sequenceable~\cite{MP}, hence doubly sequenceable by Lemma~\ref{lem:single-to-double}. Since $2\Sig G=0$ by Lemma~\ref{lem:sigma}(a), Lemma~\ref{lem:double} applied to $W=G\setminus\{1_G\}$ (in parts (a) and (b)) and to $W=G$ (in part (a)) gives the remaining three properties.
\end{proof}

\begin{corollary}
Let $G$ be a sufficiently large group of order $n$ that is not an elementary abelian $2$-group, and let $X\subseteq G$ with $|X|\le\sqrt{n/2}-1$. Then $G\setminus X$ is doubly harmonious if and only if $2\Sig X=0$; and if $1_G\in X$ and $|X|\leq \sqrt{n/2}-1$, then $G\setminus X$ is doubly R-sequenceable if and only if $2\Sig X=0$.
\end{corollary}

\begin{proof}
Apply Lemma~\ref{lem:double} to $W=G\setminus X$, using $2\Sig W=2\Sig G-2\Sig X=-2\Sig X$.
\end{proof}

%=====================================================================
\section{Small groups}\label{sec:small}
%=====================================================================

Conjecture~\ref{conj:double} is open for small non-abelian groups that are neither sequenceable nor satisfy the Hall--Paige condition. A computer search gives the following.

\begin{proposition}\label{prop:small}
Conjecture~\ref{conj:double} holds for every group $G$ with $4\le|G|\le 12$, for the dihedral groups $D_{14}$, $D_{18}$, $D_{22}$, and for the dicyclic groups $\Dic_{20}$, $\Dic_{28}$, $\Dic_{36}$.
\end{proposition}

\begin{proof}
For abelian groups this follows from~\cite{Jav24}, Theorem~\ref{thm:abelianDR} and~\cite{DJ}. The non-abelian groups of order at most $12$ are $D_6$, $D_8$, $Q_8=\Dic_8$, $D_{10}$, $A_4$, $D_{12}$ and $\Dic_{12}$. For each group in the statement, Table~\ref{tab:small} lists either a doubly sequenceable arrangement (type DS) or the partial products of a sequencing (type S, which yields a doubly sequenceable arrangement by Lemma~\ref{lem:single-to-double}), together with a doubly R-sequenceable (DR), a doubly harmonious (DH), and a doubly $\#$-harmonious (D\#H) cyclic sequence. Each entry was verified by computer; the verification is elementary and can be repeated by hand.
\end{proof}

Note that $D_{4n+2}$ and $\Dic_{4m}$ ($m$ odd) fail the Hall--Paige condition, so for these groups the double properties are not consequences of single ones via Lemma~\ref{lem:single-to-double}.

\appendix
\section{Tables}

In the tables below, $D_{2r}=\langle\rho,\sigma:\rho^r=\sigma^2=1,\ \sigma\rho\sigma=\rho^{-1}\rangle$ and $i$, $i^{*}$ denote $\rho^i$ and $\rho^i\sigma$. For $\Dic_{4m}=\langle a,b: a^{2m}=1,\ b^2=a^m,\ b^{-1}ab=a^{-1}\rangle$ (with $Q_8=\Dic_8$), $i$ and $i^{*}$ denote $a^i$ and $a^ib$. Elements of $A_4$ are written in cycle notation on $\{1,2,3,4\}$, with products composed from right to left, and $e$ is the identity. A row of type S lists the partial products $b_1=1,b_2,\dots,b_n$ of a sequencing; a row of type DS lists $\bfg\in P_2(G)$ with $g_1=1_G$ and $\bar\bfg\in P_2(G)$; rows of types DR, DH and D\#H list cyclic sequences $\bfg$ with $\bfg,\check\bfg\in P_2(G\setminus\{1_G\})$, with $\bfg,\hat\bfg\in P_2(G)$, and with $\bfg,\hat\bfg\in P_2(G\setminus\{1_G\})$, respectively. (Consecutive repeated terms are allowed in DH and D\#H sequences.)

{\small
\begin{longtable}{>{$}l<{$} l >{\raggedright\arraybackslash}p{0.76\textwidth}}
\caption{Double sequences in small groups.}\label{tab:small}\\
\toprule
G & type & sequence\\
\midrule
\endfirsthead
\toprule
G & type & sequence\\
\midrule
\endhead
D_6 & DS & $0$, $0$, $1$, $0^{*}$, $1$, $1^{*}$, $2^{*}$, $1^{*}$, $2^{*}$, $2$, $0^{*}$, $2$\\[2pt]
D_6 & DR & $1$, $0^{*}$, $2$, $1$, $0^{*}$, $2$, $2^{*}$, $1^{*}$, $2^{*}$, $1^{*}$\\[2pt]
D_6 & DH & $0$, $1^{*}$, $1^{*}$, $2$, $1$, $1$, $0^{*}$, $2^{*}$, $2$, $0$, $0^{*}$, $2^{*}$\\[2pt]
D_6 & D\#H & $1$, $0^{*}$, $1$, $2^{*}$, $1^{*}$, $2$, $2$, $2^{*}$, $0^{*}$, $1^{*}$\\[2pt]
D_8 & DS & $0$, $0$, $1$, $2$, $0^{*}$, $1$, $1^{*}$, $3$, $0^{*}$, $2^{*}$, $3$, $3^{*}$, $1^{*}$, $2^{*}$, $3^{*}$, $2$\\[2pt]
D_8 & DR & $1$, $1^{*}$, $2$, $3^{*}$, $3$, $2$, $3$, $1^{*}$, $2^{*}$, $0^{*}$, $3^{*}$, $1$, $2^{*}$, $0^{*}$\\[2pt]
D_8 & DH & $0$, $0^{*}$, $1$, $1^{*}$, $3$, $0$, $2$, $1^{*}$, $2^{*}$, $1$, $3$, $2$, $3^{*}$, $3^{*}$, $2^{*}$, $0^{*}$\\[2pt]
D_8 & D\#H & $1$, $1^{*}$, $2$, $3^{*}$, $3$, $2$, $3$, $1^{*}$, $2^{*}$, $0^{*}$, $1$, $0^{*}$, $2^{*}$, $3^{*}$\\[2pt]
Q_8 & DS & $0$, $0$, $1$, $2$, $0^{*}$, $1$, $0^{*}$, $3$, $3^{*}$, $1^{*}$, $2^{*}$, $3$, $3^{*}$, $1^{*}$, $2^{*}$, $2$\\[2pt]
Q_8 & DR & $1$, $1^{*}$, $2$, $3^{*}$, $3$, $2$, $3$, $1^{*}$, $0^{*}$, $2^{*}$, $3^{*}$, $1$, $0^{*}$, $2^{*}$\\[2pt]
Q_8 & DH & $0$, $0^{*}$, $1$, $1^{*}$, $3$, $0$, $2$, $1^{*}$, $2^{*}$, $0^{*}$, $3$, $1$, $3^{*}$, $2^{*}$, $3^{*}$, $2$\\[2pt]
Q_8 & D\#H & $1$, $1^{*}$, $2$, $3^{*}$, $3$, $2$, $3$, $1^{*}$, $0^{*}$, $0^{*}$, $3^{*}$, $1$, $2^{*}$, $2^{*}$\\[2pt]
D_{10} & S & $0$, $2^{*}$, $0^{*}$, $2$, $1^{*}$, $1$, $4$, $3$, $4^{*}$, $3^{*}$\\[2pt]
D_{10} & DR & $1$, $3^{*}$, $3$, $1$, $4^{*}$, $2$, $3$, $2^{*}$, $2$, $1^{*}$, $0^{*}$, $1^{*}$, $3^{*}$, $4^{*}$, $2^{*}$, $4$, $0^{*}$, $4$\\[2pt]
D_{10} & DH & $0$, $2^{*}$, $4^{*}$, $2^{*}$, $0$, $4$, $4^{*}$, $4$, $3$, $3^{*}$, $2$, $3$, $2$, $1^{*}$, $1$, $3^{*}$, $0^{*}$, $1^{*}$, $0^{*}$, $1$\\[2pt]
D_{10} & D\#H & $1$, $3^{*}$, $3$, $3$, $4^{*}$, $1$, $4^{*}$, $2^{*}$, $3^{*}$, $0^{*}$, $1^{*}$, $2$, $1^{*}$, $4$, $2^{*}$, $0^{*}$, $4$, $2$\\[2pt]
A_4 & DS & $e$, $e$, $(234)$, $(243)$, $(234)$, $(12)(34)$, $(243)$, $(12)(34)$, $(132)$, $(123)$, $(124)$, $(13)(24)$, $(143)$, $(124)$, $(142)$, $(123)$, $(14)(23)$, $(134)$, $(142)$, $(134)$, $(143)$, $(132)$, $(14)(23)$, $(13)(24)$\\[2pt]
A_4 & DR & $(234)$, $(142)$, $(132)$, $(13)(24)$, $(243)$, $(142)$, $(12)(34)$, $(243)$, $(234)$, $(134)$, $(14)(23)$, $(143)$, $(123)$, $(134)$, $(124)$, $(132)$, $(143)$, $(13)(24)$, $(123)$, $(14)(23)$, $(12)(34)$, $(124)$\\[2pt]
A_4 & DH & $e$, $(234)$, $(124)$, $e$, $(123)$, $(134)$, $(124)$, $(234)$, $(142)$, $(12)(34)$, $(13)(24)$, $(243)$, $(134)$, $(143)$, $(142)$, $(132)$, $(143)$, $(14)(23)$, $(14)(23)$, $(243)$, $(123)$, $(13)(24)$, $(12)(34)$, $(132)$\\[2pt]
A_4 & D\#H & $(234)$, $(142)$, $(132)$, $(13)(24)$, $(243)$, $(124)$, $(132)$, $(243)$, $(124)$, $(234)$, $(13)(24)$, $(134)$, $(12)(34)$, $(134)$, $(142)$, $(143)$, $(14)(23)$, $(143)$, $(12)(34)$, $(123)$, $(123)$, $(14)(23)$\\[2pt]
D_{12} & S & $0$, $4^{*}$, $2^{*}$, $1^{*}$, $5$, $0^{*}$, $1$, $4$, $3$, $3^{*}$, $5^{*}$, $2$\\[2pt]
D_{12} & DR & $1$, $3^{*}$, $0^{*}$, $2^{*}$, $2$, $5$, $0^{*}$, $2$, $1$, $3^{*}$, $1^{*}$, $5^{*}$, $4$, $3$, $4$, $1^{*}$, $3$, $2^{*}$, $4^{*}$, $5$, $5^{*}$, $4^{*}$\\[2pt]
D_{12} & DH & $0$, $1$, $5$, $0$, $4$, $1^{*}$, $5$, $1$, $3^{*}$, $3$, $2^{*}$, $2$, $1^{*}$, $0^{*}$, $4^{*}$, $0^{*}$, $4$, $5^{*}$, $2^{*}$, $5^{*}$, $3^{*}$, $2$, $3$, $4^{*}$\\[2pt]
D_{12} & D\#H & $1$, $3^{*}$, $0^{*}$, $2^{*}$, $2$, $5$, $0^{*}$, $2$, $5$, $3^{*}$, $3$, $4^{*}$, $5^{*}$, $3$, $4^{*}$, $2^{*}$, $5^{*}$, $1^{*}$, $4$, $1^{*}$, $4$, $1$\\[2pt]
\mathrm{Dic}_{12} & S & $0$, $1$, $3^{*}$, $4^{*}$, $0^{*}$, $5$, $2$, $2^{*}$, $4$, $1^{*}$, $5^{*}$, $3$\\[2pt]
\mathrm{Dic}_{12} & DR & $1$, $3^{*}$, $0^{*}$, $2^{*}$, $2$, $5$, $0^{*}$, $2$, $1$, $3^{*}$, $1^{*}$, $5$, $2^{*}$, $1^{*}$, $5^{*}$, $4$, $4^{*}$, $3$, $4$, $4^{*}$, $5^{*}$, $3$\\[2pt]
\mathrm{Dic}_{12} & DH & $0$, $1$, $5$, $0$, $4$, $1^{*}$, $5$, $1$, $3^{*}$, $3$, $2^{*}$, $2$, $1^{*}$, $0^{*}$, $0^{*}$, $4^{*}$, $3$, $4$, $3^{*}$, $4^{*}$, $5^{*}$, $5^{*}$, $2$, $2^{*}$\\[2pt]
\mathrm{Dic}_{12} & D\#H & $1$, $3^{*}$, $5$, $2^{*}$, $2$, $5$, $1^{*}$, $2$, $0^{*}$, $0^{*}$, $3$, $2^{*}$, $1^{*}$, $5^{*}$, $4^{*}$, $4^{*}$, $3$, $4$, $4$, $5^{*}$, $3^{*}$, $1$\\[2pt]
D_{14} & S & $0$, $0^{*}$, $2$, $3$, $1$, $3^{*}$, $6$, $2^{*}$, $6^{*}$, $4^{*}$, $5$, $4$, $5^{*}$, $1^{*}$\\[2pt]
D_{14} & DR & $1$, $2$, $3^{*}$, $1$, $4^{*}$, $2$, $5^{*}$, $3^{*}$, $6$, $5$, $3$, $6$, $4$, $1^{*}$, $3$, $2^{*}$, $0^{*}$, $6^{*}$, $5$, $5^{*}$, $1^{*}$, $4^{*}$, $4$, $2^{*}$, $6^{*}$, $0^{*}$\\[2pt]
D_{14} & DH & $0$, $1$, $3$, $1$, $4^{*}$, $6$, $1^{*}$, $2^{*}$, $6$, $2^{*}$, $1^{*}$, $3^{*}$, $3$, $4$, $2$, $5$, $4$, $5^{*}$, $0^{*}$, $5$, $5^{*}$, $3^{*}$, $0^{*}$, $4^{*}$, $0$, $6^{*}$, $2$, $6^{*}$\\[2pt]
D_{14} & D\#H & $1$, $2$, $3^{*}$, $1$, $4^{*}$, $2$, $5^{*}$, $3^{*}$, $6$, $6$, $4$, $4$, $4^{*}$, $2^{*}$, $5^{*}$, $1^{*}$, $5$, $6^{*}$, $0^{*}$, $2^{*}$, $3$, $3$, $0^{*}$, $6^{*}$, $5$, $1^{*}$\\[2pt]
D_{18} & S & $0$, $4^{*}$, $2^{*}$, $4$, $2$, $7$, $8^{*}$, $5$, $6$, $3$, $5^{*}$, $8$, $7^{*}$, $3^{*}$, $0^{*}$, $1^{*}$, $1$, $6^{*}$\\[2pt]
D_{18} & DR & $1$, $3$, $8$, $8^{*}$, $2^{*}$, $8$, $2$, $5^{*}$, $7$, $5$, $0^{*}$, $2$, $8^{*}$, $6^{*}$, $2^{*}$, $3$, $4$, $0^{*}$, $4$, $1$, $5^{*}$, $1^{*}$, $3^{*}$, $4^{*}$, $7$, $6$, $7^{*}$, $6$, $6^{*}$, $1^{*}$, $7^{*}$, $5$, $4^{*}$, $3^{*}$\\[2pt]
D_{18} & DH & $0$, $2^{*}$, $0$, $5$, $8^{*}$, $2$, $6^{*}$, $8$, $8$, $3$, $7$, $0^{*}$, $8^{*}$, $6^{*}$, $7$, $1$, $4^{*}$, $1^{*}$, $2^{*}$, $7^{*}$, $2$, $4$, $5^{*}$, $5$, $1$, $5^{*}$, $4$, $0^{*}$, $6$, $3$, $6$, $4^{*}$, $1^{*}$, $3^{*}$, $7^{*}$, $3^{*}$\\[2pt]
D_{18} & D\#H & $1$, $3$, $8$, $8^{*}$, $2^{*}$, $8$, $2$, $5^{*}$, $1^{*}$, $5$, $0^{*}$, $7^{*}$, $6^{*}$, $2$, $8^{*}$, $3^{*}$, $3$, $5^{*}$, $0^{*}$, $7$, $6^{*}$, $5$, $7^{*}$, $4^{*}$, $7$, $1^{*}$, $4^{*}$, $4$, $4$, $2^{*}$, $3^{*}$, $1$, $6$, $6$\\[2pt]
D_{22} & S & $0$, $2^{*}$, $0^{*}$, $7^{*}$, $1^{*}$, $8$, $2$, $10^{*}$, $3^{*}$, $3$, $6$, $9^{*}$, $10$, $8^{*}$, $7$, $5$, $4$, $1$, $6^{*}$, $5^{*}$, $9$, $4^{*}$\\[2pt]
D_{22} & DR & $1$, $0^{*}$, $2$, $3^{*}$, $3$, $4^{*}$, $4$, $10$, $3^{*}$, $8^{*}$, $4^{*}$, $0^{*}$, $8^{*}$, $7^{*}$, $9$, $8$, $10$, $8$, $2^{*}$, $9$, $7$, $6^{*}$, $10^{*}$, $5$, $1$, $4$, $1^{*}$, $10^{*}$, $5^{*}$, $2$, $7$, $6$, $2^{*}$, $5^{*}$, $3$, $9^{*}$, $6$, $1^{*}$, $5$, $7^{*}$, $6^{*}$, $9^{*}$\\[2pt]
D_{22} & DH & $0$, $10$, $1$, $5^{*}$, $1$, $4^{*}$, $4$, $5^{*}$, $2^{*}$, $8^{*}$, $3^{*}$, $2$, $0^{*}$, $7^{*}$, $2$, $7^{*}$, $3$, $9$, $4^{*}$, $2^{*}$, $3$, $5$, $4$, $8^{*}$, $0$, $6$, $6$, $0^{*}$, $8$, $5$, $10$, $8$, $3^{*}$, $7$, $1^{*}$, $6^{*}$, $9^{*}$, $10^{*}$, $10^{*}$, $1^{*}$, $9$, $9^{*}$, $6^{*}$, $7$\\[2pt]
D_{22} & D\#H & $1$, $0^{*}$, $2$, $3^{*}$, $3$, $4^{*}$, $4$, $5^{*}$, $3^{*}$, $8^{*}$, $4^{*}$, $0^{*}$, $8^{*}$, $7^{*}$, $3$, $9$, $7$, $10$, $5^{*}$, $9$, $10^{*}$, $7^{*}$, $2^{*}$, $8$, $1$, $1^{*}$, $10$, $9^{*}$, $2^{*}$, $7$, $10^{*}$, $1^{*}$, $2$, $6$, $4$, $6^{*}$, $5$, $5$, $8$, $6^{*}$, $9^{*}$, $6$\\[2pt]
\mathrm{Dic}_{20} & S & $0$, $5$, $3^{*}$, $1$, $4$, $8^{*}$, $0^{*}$, $1^{*}$, $4^{*}$, $7$, $9$, $2^{*}$, $6$, $6^{*}$, $5^{*}$, $9^{*}$, $8$, $2$, $7^{*}$, $3$\\[2pt]
\mathrm{Dic}_{20} & DR & $1$, $0^{*}$, $3^{*}$, $1^{*}$, $8$, $4^{*}$, $8$, $8^{*}$, $9$, $6$, $5$, $2^{*}$, $6^{*}$, $7^{*}$, $7$, $1^{*}$, $9$, $9^{*}$, $2$, $5^{*}$, $4$, $6$, $8^{*}$, $4^{*}$, $9^{*}$, $6^{*}$, $5^{*}$, $0^{*}$, $7^{*}$, $3$, $4$, $2^{*}$, $2$, $3^{*}$, $5$, $3$, $1$, $7$\\[2pt]
\mathrm{Dic}_{20} & DH & $0$, $0^{*}$, $3^{*}$, $8^{*}$, $7$, $3^{*}$, $7$, $8^{*}$, $9$, $5^{*}$, $1^{*}$, $5^{*}$, $1$, $0$, $1^{*}$, $6$, $2$, $4$, $3$, $6^{*}$, $9$, $9^{*}$, $3$, $5$, $8$, $9^{*}$, $4^{*}$, $6$, $8$, $4^{*}$, $7^{*}$, $4$, $2$, $1$, $2^{*}$, $2^{*}$, $0^{*}$, $6^{*}$, $7^{*}$, $5$\\[2pt]
\mathrm{Dic}_{20} & D\#H & $1$, $0^{*}$, $3^{*}$, $1^{*}$, $8$, $4^{*}$, $8$, $8^{*}$, $9$, $6$, $5$, $2^{*}$, $6^{*}$, $7^{*}$, $7$, $1^{*}$, $9$, $6$, $2^{*}$, $1$, $3$, $4^{*}$, $0^{*}$, $8^{*}$, $5^{*}$, $7^{*}$, $7$, $2$, $4$, $9^{*}$, $5$, $3$, $6^{*}$, $9^{*}$, $4$, $2$, $3^{*}$, $5^{*}$\\[2pt]
\mathrm{Dic}_{28} & S & $0$, $8$, $2^{*}$, $0^{*}$, $1^{*}$, $2$, $13^{*}$, $5$, $3^{*}$, $7^{*}$, $4^{*}$, $12$, $12^{*}$, $3$, $8^{*}$, $6$, $10$, $1$, $11^{*}$, $5^{*}$, $10^{*}$, $9^{*}$, $13$, $6^{*}$, $9$, $7$, $4$, $11$\\[2pt]
\mathrm{Dic}_{28} & DR & $1$, $4$, $11^{*}$, $13^{*}$, $1^{*}$, $10^{*}$, $9$, $13$, $1^{*}$, $8^{*}$, $2$, $10^{*}$, $7^{*}$, $8$, $7^{*}$, $11$, $10$, $12$, $5$, $2$, $12$, $8^{*}$, $13$, $10$, $11^{*}$, $5^{*}$, $0^{*}$, $7$, $0^{*}$, $11$, $6^{*}$, $8$, $5^{*}$, $9^{*}$, $3^{*}$, $4^{*}$, $9^{*}$, $1$, $13^{*}$, $3$, $12^{*}$, $4^{*}$, $6$, $3^{*}$, $2^{*}$, $3$, $7$, $9$, $4$, $2^{*}$, $5$, $6$, $6^{*}$, $12^{*}$\\[2pt]
\mathrm{Dic}_{28} & DH & $0$, $3$, $12$, $7^{*}$, $13$, $2^{*}$, $4$, $2^{*}$, $10^{*}$, $9^{*}$, $8^{*}$, $4$, $3^{*}$, $13$, $8$, $9^{*}$, $10^{*}$, $0^{*}$, $0$, $12^{*}$, $8^{*}$, $12$, $6$, $11^{*}$, $5^{*}$, $6$, $3^{*}$, $0^{*}$, $9$, $11$, $5^{*}$, $12^{*}$, $1$, $6^{*}$, $4^{*}$, $13^{*}$, $7$, $6^{*}$, $11^{*}$, $3$, $9$, $5$, $10$, $13^{*}$, $1^{*}$, $1$, $10$, $1^{*}$, $5$, $2$, $7$, $11$, $8$, $7^{*}$, $4^{*}$, $2$\\[2pt]
\mathrm{Dic}_{28} & D\#H & $1$, $4$, $8^{*}$, $3$, $1^{*}$, $4^{*}$, $7$, $0^{*}$, $1^{*}$, $12^{*}$, $4$, $11^{*}$, $6^{*}$, $6^{*}$, $12^{*}$, $7^{*}$, $12$, $6$, $10$, $11$, $6$, $9$, $10$, $1$, $2$, $9$, $13$, $8^{*}$, $3$, $3^{*}$, $2$, $7^{*}$, $4^{*}$, $3^{*}$, $11^{*}$, $9^{*}$, $5$, $5^{*}$, $5$, $9^{*}$, $10^{*}$, $8$, $2^{*}$, $13$, $0^{*}$, $5^{*}$, $13^{*}$, $7$, $10^{*}$, $12$, $13^{*}$, $11$, $2^{*}$, $8$\\[2pt]
\mathrm{Dic}_{36} & S & $0$, $9$, $8$, $15^{*}$, $11$, $17$, $3^{*}$, $10$, $2^{*}$, $12^{*}$, $11^{*}$, $17^{*}$, $15$, $14^{*}$, $4$, $10^{*}$, $14$, $7$, $7^{*}$, $9^{*}$, $2$, $16^{*}$, $1^{*}$, $13$, $16$, $12$, $1$, $3$, $6^{*}$, $6$, $0^{*}$, $5^{*}$, $13^{*}$, $8^{*}$, $4^{*}$, $5$\\[2pt]
\mathrm{Dic}_{36} & DR & $1$, $16$, $1$, $16$, $11$, $7$, $7^{*}$, $12$, $5$, $2^{*}$, $1^{*}$, $13^{*}$, $10^{*}$, $4^{*}$, $12$, $9^{*}$, $10^{*}$, $2$, $0^{*}$, $9$, $3^{*}$, $2$, $0^{*}$, $14$, $3^{*}$, $8$, $16^{*}$, $4^{*}$, $11^{*}$, $8$, $13^{*}$, $17$, $9$, $15^{*}$, $4$, $14$, $15^{*}$, $11^{*}$, $12^{*}$, $17^{*}$, $1^{*}$, $4$, $17^{*}$, $6$, $14^{*}$, $9^{*}$, $5^{*}$, $14^{*}$, $6$, $7$, $16^{*}$, $15$, $11$, $13$, $3$, $6^{*}$, $12^{*}$, $7^{*}$, $5^{*}$, $5$, $13$, $6^{*}$, $8^{*}$, $10$, $17$, $2^{*}$, $15$, $8^{*}$, $3$, $10$\\[2pt]
\mathrm{Dic}_{36} & DH & $0$, $0$, $15^{*}$, $6^{*}$, $5^{*}$, $5$, $7^{*}$, $14$, $10^{*}$, $7^{*}$, $8$, $13^{*}$, $2$, $10^{*}$, $11^{*}$, $9^{*}$, $12$, $2^{*}$, $1^{*}$, $4^{*}$, $5^{*}$, $9$, $2$, $6^{*}$, $17^{*}$, $13$, $3^{*}$, $14^{*}$, $5$, $1$, $14$, $9$, $11^{*}$, $17^{*}$, $3^{*}$, $4$, $16^{*}$, $13^{*}$, $0^{*}$, $17$, $3$, $4^{*}$, $12^{*}$, $3$, $4$, $10$, $15$, $1^{*}$, $8^{*}$, $0^{*}$, $13$, $8^{*}$, $8$, $6$, $16$, $15$, $16^{*}$, $6$, $7$, $15^{*}$, $9^{*}$, $17$, $2^{*}$, $12$, $12^{*}$, $7$, $10$, $11$, $16$, $11$, $14^{*}$, $1$\\[2pt]
\mathrm{Dic}_{36} & D\#H & $1$, $8$, $9^{*}$, $6^{*}$, $5^{*}$, $5$, $7^{*}$, $14$, $10^{*}$, $7^{*}$, $8$, $15^{*}$, $10$, $5$, $9^{*}$, $12$, $3^{*}$, $6^{*}$, $15$, $1^{*}$, $12^{*}$, $11$, $13^{*}$, $17^{*}$, $2^{*}$, $2^{*}$, $14^{*}$, $14$, $14^{*}$, $3^{*}$, $4^{*}$, $12^{*}$, $17^{*}$, $1^{*}$, $11^{*}$, $16$, $10^{*}$, $2$, $6$, $17$, $4^{*}$, $13$, $0^{*}$, $2$, $0^{*}$, $7$, $15^{*}$, $13^{*}$, $12$, $16$, $3$, $11$, $3$, $10$, $1$, $6$, $15$, $5^{*}$, $9$, $7$, $13$, $8^{*}$, $4$, $9$, $16^{*}$, $4$, $17$, $11^{*}$, $16^{*}$, $8^{*}$\\[2pt]
\bottomrule
\end{longtable}
}

{\small
\begin{longtable}{>{$}l<{$} >{\raggedright\arraybackslash}p{0.85\textwidth}}
\caption{Partial products $b_0=1,b_1=b,\dots,b_{4m-1}=a^m$ of sequencings of $\Dic_{4m}$ used in Proposition~\ref{prop:dic}.}\label{tab:dic}\\
\toprule
m & partial products\\
\midrule
\endfirsthead
\toprule
m & partial products\\
\midrule
\endhead
3 & $0$, $0^{*}$, $3^{*}$, $1$, $2$, $4$, $5^{*}$, $5$, $1^{*}$, $2^{*}$, $4^{*}$, $3$\\[2pt]
5 & $0$, $0^{*}$, $3^{*}$, $1$, $4$, $8^{*}$, $6^{*}$, $3$, $2^{*}$, $1^{*}$, $5^{*}$, $9$, $8$, $2$, $7^{*}$, $6$, $9^{*}$, $4^{*}$, $7$, $5$\\[2pt]
7 & $0$, $0^{*}$, $2^{*}$, $13$, $6^{*}$, $2$, $8^{*}$, $13^{*}$, $12$, $1^{*}$, $7^{*}$, $10$, $9^{*}$, $10^{*}$, $3^{*}$, $5$, $8$, $4$, $1$, $6$, $4^{*}$, $9$, $11$, $12^{*}$, $11^{*}$, $5^{*}$, $3$, $7$\\[2pt]
9 & $0$, $0^{*}$, $3$, $10^{*}$, $14^{*}$, $7^{*}$, $2$, $4$, $15^{*}$, $12$, $13$, $3^{*}$, $15$, $6^{*}$, $11$, $14$, $1$, $17^{*}$, $7$, $17$, $2^{*}$, $12^{*}$, $8^{*}$, $11^{*}$, $16^{*}$, $5$, $16$, $10$, $9^{*}$, $8$, $6$, $1^{*}$, $13^{*}$, $4^{*}$, $5^{*}$, $9$\\[2pt]
11 & $0$, $0^{*}$, $16^{*}$, $2$, $21$, $13^{*}$, $6^{*}$, $5^{*}$, $19$, $3^{*}$, $1^{*}$, $18^{*}$, $4^{*}$, $17^{*}$, $1$, $21^{*}$, $10^{*}$, $17$, $8$, $7^{*}$, $12^{*}$, $2^{*}$, $6$, $4$, $19^{*}$, $11^{*}$, $13$, $16$, $10$, $20^{*}$, $18$, $8^{*}$, $3$, $15$, $14$, $15^{*}$, $9$, $5$, $20$, $9^{*}$, $12$, $14^{*}$, $7$, $11$\\[2pt]
13 & $0$, $0^{*}$, $3$, $5^{*}$, $19^{*}$, $20$, $10$, $24$, $25$, $18$, $23^{*}$, $20^{*}$, $19$, $15$, $6^{*}$, $4$, $7^{*}$, $15^{*}$, $7$, $11$, $6$, $4^{*}$, $10^{*}$, $1$, $21^{*}$, $23$, $1^{*}$, $21$, $2^{*}$, $3^{*}$, $24^{*}$, $18^{*}$, $8$, $9^{*}$, $16$, $14$, $5$, $12$, $22$, $12^{*}$, $25^{*}$, $17^{*}$, $17$, $2$, $11^{*}$, $22^{*}$, $13^{*}$, $16^{*}$, $14^{*}$, $9$, $8^{*}$, $13$\\[2pt]
\bottomrule
\end{longtable}
}

\end{document}